\documentclass[reqno, 11pt]{article}

 \usepackage{amsthm, amssymb, amsmath}
\usepackage[colorlinks=true]{hyperref}
\usepackage{fullpage}
\usepackage{verbatim}
\usepackage{graphicx}
\usepackage{epsfig}
\usepackage{color}
\usepackage[all]{xy}

\newtheorem{thm}{Theorem}[section]
\newtheorem*{thm*}{Theorem}
\newtheorem{lem}[thm]{Lemma}
\newtheorem{cor}[thm]{Corollary}
\newtheorem{prop}[thm]{Proposition}

\newtheorem*{conjecture*}{Conjecture}

\theoremstyle{remark} 
\newtheorem*{question*}{Question}

\newtheorem{remark}[thm]{Remark}

\theoremstyle{definition} 
\newtheorem{define}[thm]{Definition}

\numberwithin{equation}{section}  

\newcommand{\OO}{\mathcal{O}}    
\newcommand{\FF}{\mathbb{F}}      
\newcommand{\RR}{\mathbb{R}}     
\newcommand{\PP}{\mathbb{P}}      
\newcommand{\QQ}{\mathbb{Q}}      
\newcommand{\CC}{\mathbb{C}}      

\newcommand{\be}{\begin{equation}}
\newcommand{\ee}{\end{equation}}
\newcommand{\benn}{\begin{equation*}}
\newcommand{\eenn}{\end{equation*}}
\newcommand{\ba}{\begin{aligned}}
\newcommand{\ea}{\end{aligned}}
\newcommand{\bbm}{\begin{bmatrix}}
\newcommand{\ebm}{\end{bmatrix}}
\newcommand{\bpm}{\begin{pmatrix}}
\newcommand{\epm}{\end{pmatrix}}
\newcommand{\bi}{\begin{itemize}}
\newcommand{\ei}{\end{itemize}}
 
\newcommand{\ord}{\operatorname{ord}}

\newcommand{\simarrow}{\stackrel{\sim}{\rightarrow}}    

\newcommand{\Berk}{\mathbf{P}^1}  

\newcommand{\Hull}{\mathrm{Hull}}      
\newcommand{\Crit}{\mathrm{Crit}}     
\newcommand{\Ram}{\mathcal{R}}      
\newcommand{\PGL}{\mathrm{PGL}}    
\newcommand{\DD}{\mathcal{D}}     
\newcommand{\BB}{\mathcal{B}}    
\newcommand{\HH}{\mathbf{H}}     
\newcommand{\Ramtot}{\mathcal{R}^{\mathrm{tot}}}   
\renewcommand{\aa}{\mathbf{a}}
\newcommand{\bb}{\mathbf{b}}
\newcommand{\rr}{\mathbf{r}}
\renewcommand{\ss}{\mathbf{s}}
\newcommand{\Aff}{\mathbf{A}}
\newcommand{\diam}{\mathrm{diam}}
\renewcommand{\AA}{\mathcal{A}}

\newcommand{\wronsk}{\mathrm{Wr}} 

\newtheorem*{thmA}{Theorem A}
\newtheorem*{thmB}{Theorem B}
\newtheorem*{thmC}{Theorem C}

\title{Topology and Geometry of the Berkovich Ramification Locus for
	Rational Functions, I}

\author{Xander Faber \\
Department of Mathematics \\
University of Hawaii, Honolulu, HI  \\
\texttt{xander@math.hawaii.edu}
}

\date{}

\begin{document}

\maketitle
\tableofcontents


	\begin{abstract}
		We initiate a detailed study of 
		the ramification locus for projective endomorphisms of the Berkovich projective line --- 
		the non-Archimedean analog of the Riemann sphere.  
	\end{abstract}

\section{Introduction}

	Given a nonconstant holomorphic map $f: X \to Y$ between compact Riemann surfaces, one of the first objects we learn to construct is its ramification divisor $R_f$, which describes the locus at which $f$ fails to be locally injective. The divisor $R_f$ is a formal linear combination of points of $X$ that is combinatorially constrained by the Hurwitz Formula: $2g_X - 2 = \deg(f) (2g_Y - 2) + \deg(R_f)$.
	
	The goal of the present article is to initiate a study of the ramification locus in the setting of non-Archimedean analytic geometry. Here the role of a Riemann surface is played by a projective Berkovich analytic curve over a non-Archimedean field $k$. As these curves have many points that are not algebraic over $k$, some new (non-algebraic) ramification behavior appears. For example, the ramification locus is no longer a divisor, but rather a closed analytic subspace. Berkovich first observed this ``geometric ramification'' in \cite[\S6.3]{Berkovich_Etale_1993}. 
	
	We begin our study by restricting attention to rational functions, viewed as endomorphisms of the projective line $\Berk$. This simplest first case  has the benefit of being approachable by concrete techniques, many of which were developed by Rivera-Letelier \cite{Rivera-Letelier_Asterisque_2003, Rivera-Letelier_Espace_Hyperbolique_2003, Rivera-Letelier_Periodic_Points_2005}, Favre/Rivera-Letelier \cite{Favre_Rivera-Letelier_Ergodic_2010},  and Baker/Rumely \cite{Baker-Rumely_BerkBook_2010}. As critical points occupy a central position in the study of complex dynamical systems on the Riemann sphere, it is not unreasonable to suppose that a better understanding of the Berkovich ramification locus for rational functions will have applications to non-Archimedean dynamical systems. In fact, this work was initially inspired by dynamical considerations in \cite{Favre_Rivera-Letelier_Ergodic_2010}. The simple structure of the ramification locus for tame polynomials plays a fundamental role in the recent work of Trucco \cite{Trucco_Tame_Polynomials_2012}. The nature of the ramification locus for dynamical systems defined over the formal Laurent series field $\CC(\!(t)\!)$ also sheds some light on degenerations of complex dynamical systems. See \cite{Kiwi_Puiseux_Dynamics_2006, Kiwi_Rescaling_Limits_2012}.

	Let $k$ be an algebraically closed field that is complete with respect to a fixed nontrivial non-Archimedean absolute value $|\cdot|$. For example, $k$ could be the completion of an algebraic closure of $\QQ_p$ or of $\FF_p(\!(t)\!)$. Write $\PP^1_k$ for the (algebraic) projective line over $k$, and write $\Berk = \Berk_k$ for its Berkovich analytification. 
	A rational function $\varphi \in k(z)$, viewed as a morphism $\varphi: \PP^1_k \to \PP^1_k$, extends functorially to a morphism of $\Berk$ (which we also call $\varphi$). Intuitively, it describes the action of $\varphi$ on disks in $\PP^1(k)$. 
As $\varphi$ is a finite morphism, one may associate to each point $x \in \Berk$ a local degree or multiplicity $m_\varphi(x)$: in a weak neighborhood of $x$, the map $\varphi$ is $m_\varphi(x)$-to-$1$. The \textbf{Berkovich ramification locus} is defined to be the set 
	\[
		\Ram_\varphi = \{x \in \Berk : m_\varphi(x) > 1\}.
	\] 
It is a closed subset of $\Berk$ with no isolated point. 
Our first main result provides a bound for the number of connected components $\Ram_\varphi$. 

\begin{thmA}[Connected Components]
	Let $\varphi \in k(z)$ be a nonconstant rational function. Each connected component of the Berkovich ramification locus of $\varphi$ contains at least two critical points of $\varphi$, counted with weights.\footnote{For a rational function $\varphi \in k(z)$, a point at which the induced map on the tangent space of $\PP^1_k$ vanishes will be called a \textbf{critical point}. The order of vanishing is called the \textbf{weight}.} In particular, $\Ram_\varphi$ has at most  $\deg(\varphi) - 1$ connected components.  
\end{thmA}

	The theorem is optimal in the following sense. For any algebraically closed field $k$ that is complete with respect to a nontrivial non-Archimedean absolute value and any integers $1 \leq n < d$, there exists a rational function $\varphi \in k(z)$ of degree~$d$ whose ramification locus has precisely $n$ connected components. 

	A field $k$ as above always admits nontrivial extensions by non-Archimedean valued fields; it is one feature of non-Archimedean analysis that sets it apart from complex analysis. 
Let $K / k$ be an extension of algebraically closed and complete non-Archimedean fields, so that the absolute value on $K$ is an extension of the one on $k$. There is a natural inclusion $\PP^1(k) \hookrightarrow \PP^1(K)$, and this inclusion extends to the Berkovich analytifications $\iota:\Berk_k \hookrightarrow \Berk_K$. However, this last map is not a morphism of analytic spaces (unless $K = k$!), and so we must spend some time proving that it preserves many of the features relevant to our study of ramification. In particular, we will show that this inclusion is continuous, that it preserves multiplicities, and that it preserves a certain natural metric on $\Berk \smallsetminus \PP^1(k)$. The existence of the inclusion $\iota$ is closely related to Berkovich's notion of ``peaked point'' \cite[5.2]{Berkovich_Spectral_Theory_1990} and Poineau's notion of ``universal point'' \cite{Poineau_Angelique}, although these latter notions extend to arbitrary analytic spaces. 
	
	A rational function $\varphi \in k(z)$ can act via an inseparable morphism on the local rings of certain points of $\Berk$;  Rivera-Letelier calls this ``inseparable reduction at a type~II point.''  We give a natural extension of Rivera-Letelier's definition to all points of $\Berk$ by enlarging the field $k$ in such a way that all non-classical points become type~II points. 
As an application of this work on extension of scalars and inseparable reduction, we are able to give a natural characterization of the interior of the ramification locus for the strong topology on $\Berk$. 
		
\begin{thmB}[Interior Points]
	Let $\varphi \in k(z)$ be a nonconstant rational function.
	\begin{enumerate}
		\item The set of points at which $\varphi$ has inseparable reduction coincides with the 
			strong interior of the Berkovich ramification locus. 
		\item The ramification locus has empty weak interior unless $\varphi$ is itself 
			inseparable, in which case $\Ram_\varphi = \Berk$
	\end{enumerate}
\end{thmB} 

Following Trucco \cite{Trucco_Tame_Polynomials_2012}, we say that a rational function $\varphi$ is \textbf{tame} if its ramification locus has only finitely many branch points. Theorem~B allows us to give a number of equivalent characterizations of tame rational functions (Corollary~\ref{Cor: Tame characterization}). We remark that a sufficient condition for a rational function $\varphi$ to be tame is that $\mathrm{res}.\mathrm{char}.(k) = 0$ or $\mathrm{res}.\mathrm{char}.(k) > \deg(\varphi)$ (Corollary~\ref{Cor: Large Res Char is Tame}).

	We also look at the special setting of rational functions with a totally ramified point; i.e., a point $x \in \Berk$ such that $m_\varphi(x) = \deg(\varphi)$. For example, this includes the important cases of polynomials ($x = \infty$) and rational functions with good reduction ($x$ is the Gauss point). For the following statement, let $\Hull(\Crit(\varphi))$ be the \textbf{connected hull} of the critical points; i.e., the smallest closed connected subset of $\Berk$ containing $\Crit(\varphi)$.
		
\begin{thmC}[Totally Ramified Functions]
	Let $\varphi \in k(z)$ be a nonconstant rational function for which there exists a totally ramified point in $\Berk$. Then the ramification locus $\Ram_\varphi$ is connected. In particular, if $\varphi$ is tame, then $\Ram_\varphi = \Hull(\Crit(\varphi))$. 
\end{thmC}

    In a sequel to this paper, we provide a detailed study of the geometry of the ramification locus  with respect to the hyperbolic $\PGL_2(k)$-invariant metric on $\Berk \smallsetminus \PP^1(k)$ \cite{Faber_Berk_RamII_2012}.

	We close with a detailed summary of the contents of the present paper. In Section~\ref{Sec: Conventions} we recall all of the relevant features of $\Berk$ and its endomorphisms. While this section is primarily designed to fix notation, it could also serve as a brief introduction to $\Berk$. In Section~\ref{Sec: Multiplicities} we discuss three notions of multiplicity function. The first is an extension of the algebraic multiplicity $m_\varphi$ on $\PP^1(k)$ to the entire Berkovich projective line $\Berk$. The second is the directional multiplicity, which allows one to accurately count the number of solutions to the equation $\varphi(z) = y$ in a particular open Berkovich disk $U$, provided that $\varphi(U) \neq \Berk$. It can happen that $\varphi(U) = \Berk$, and so we introduce the notion of surplus multiplicity as the defect in this counting problem. The surplus multiplicity of $U$ is very closely tied to the number of critical points contained in $U$. The first two multiplicities are well understood in the literature. This article is the first to focus on the surplus multiplicity in its own right, although it does appear in \cite[Lem.~3.2]{Rivera-Letelier_Periodic_Points_2005}. 
	
	Section~\ref{Sec: Base Change} is devoted to constructing the canonical inclusion $\iota_k^K: \Berk_k \to \Berk_K$ and proving a number of useful properties, including its compatibility with rational functions. The goal of Section~\ref{Sec: Inseparable Reduction}  is to provide a definition of inseparable reduction at an arbitrary point of $\Berk$. We also give an interesting criterion for when a rational function has inseparable reduction at a type~III point. In Section~\ref{Sec: Connected Components}, we prove Theorem~A and a number of other results related to connectedness of the ramification locus. For example, we show that every connected component of $\Ram_\varphi$ meets the convex hull of the critical points. We describe the endpoints and interior points of the ramification locus in Section~\ref{Sec: Ends and Interior}; this includes a proof of Theorem~B and a number of characterizations of tame and locally tame rational functions. Finally, in Section~8 we discuss the locus of total ramification and some of the properties of rational functions for which this locus is nonempty. 	\\


\section{Notation and Conventions}
\label{Sec: Conventions}

\subsection{Non-Archimedean Fields}

	For the duration of this paper, $k$ will denote an algebraically closed field that is complete with respect to a nontrivial non-Archimedean absolute value $|\cdot|$. 
We use the standard notation $k^{\circ} = \{ t \in k : |t| \leq 1 \}$ and $k^{\circ \circ} = \{ t \in k : |t| < 1 \}$ for the valuation ring of $k$ and for its maximal ideal, respectively, and we write $\tilde k = k^\circ / k^{\circ \circ}$ for the residue field. The residue characteristic of $k$ will be denoted~$p$. (Note $p = 0$ is allowed.) The value group of $k$ will be denoted $|k^\times|$; as $k$ is algebraically closed, $|k^\times|$ is a divisible group.
	
	The \textbf{normalized base} associated to $k$ is the constant
	\[
		q_k = \begin{cases}
			e & \text{if $k$ has equicharacteristic $p \geq 0$} \\
			|p|^{-1}  & \text{if $k$ has mixed characteristic}.				
			\end{cases}
	\]
Then $q_k > 1$, and the function $\ord_k(\cdot) = - \log_{q_k} | \cdot |$ is a valuation on $k$. 

	For $a \in k$ and $r \in \RR_{\geq 0}$, write
	\[
		D(a, r)^- = \{ x \in k : |x  - a | < r \} \quad \text{and} \quad  D(a,r) = \{x \in k : |x - a| \leq r\}
	\]
for the \textbf{(classical) open disk} and the \textbf{(classical) closed disk} of radius $r$ about $a$, respectively.

\subsection{The Berkovich Projective Line}

	Here we summarize the definition and main properties of $\Berk$. For the most part we follow the notation and treatment in \cite[\S1--2]{Baker-Rumely_BerkBook_2010}, although much of this material was first presented in \cite{Rivera-Letelier_Espace_Hyperbolique_2003, Rivera-Letelier_Periodic_Points_2005}. See also \cite[Ch.~3]{AWS_2008}. 
	
\subsubsection{The Affine Line}
	
	The Berkovich affine line $\Aff^1 = \Aff^1_k$ is defined to be the set of all multiplicative seminorms on the polynomial algebra $k[T]$ that restrict to the given absolute value on $k$. If $x$ is a seminorm and $f \in k[T]$ is a polynomial, we write $|f(x)|$ for the value of $f$ at $x$. For example, if $a \in k$ and $r \in \RR_{\geq 0}$, write $\zeta_{a,r}$ for the multiplicative seminorm defined by
	\[
		|f(\zeta_{a,r})| = \sup_{b \in D(a, r)} |f(b)|, \qquad f \in k[T].
	\]
Berkovich has classified the points of $\Aff^1$:
	\begin{enumerate}
		\item Type~I. $\zeta_{a,0}$ for some $a \in k$. (Such a point is called a \textbf{classical point}.)
		\item Type~II. $\zeta_{a,r}$ for some $a \in k$ and $r \in |k^\times|$.
		\item Type~III. $\zeta_{a,r}$ for some $a \in k$ and $r \not\in |k^\times|$.
		\item Type~IV. A limit of seminorms $(\zeta_{a_i, r_i})_{i \geq 0}$, where the associated sequence of 
			closed disks $(D(a_i, r_i))_{i \geq 0}$ is descending and has empty intersection. (The field $k$ is
			called \textbf{spherically closed} if no such sequence of closed disks exists.)
	\end{enumerate}
	
	This classification suggests a means for extending the notation to cover type~IV points. Given a decreasing sequence of closed disks $D(\aa, \rr) = (D(a_i, r_i))_{i \geq 0}$, define $\zeta_{\aa, \rr} \in \Aff^1$ to be the seminorm on $k[T]$ given by
	\[
		|f(\zeta_{ \aa, \rr})| = \lim_{i \to \infty} \sup_{b \in D(a_i, r_i)} |f(b)|, \qquad f \in k[T].
	\]
Note that $\zeta_{\aa, \rr} = \zeta_{a,r}$ if $D(\aa, \rr)$ is the constant sequence with term $D(a,r)$. More generally, 
if $\cap_{i \geq 0} D(a_i, r_i) = D(b,s)$ for some $ b\in k$ and $s \in \RR_{\geq 0}$, then one verifies easily that $\zeta_{\aa, \rr} = \zeta_{b,s}$. Moreover, we have the equality of seminorms $\zeta_{\aa,\rr} = \zeta_{\aa', \rr'}$ if and only if the associated sequences $D(\aa, \rr)$ and $D(\aa', \rr')$ are cofinal in each other.
	
	We identify the set of classical points in $\Aff^1$ with $k$ via the injection $a \mapsto \zeta_{a,0}$.  The point $\zeta_{0,1}$ is called the \textbf{Gauss point} because the associated seminorm coincides with the Gauss norm of a polynomial. 

\subsubsection{The Weak Topology}

	The \textbf{weak topology} on $\Aff^1$ is the weakest topology satisfying the following property: for each polynomial $f \in k[T]$, the function $x \mapsto |f(x)|$ is continuous on $\Aff^1$. The space $\Aff^1$ is locally compact, Hausdorff, and uniquely path-connected for the weak topology. 
		
	The injection $k \hookrightarrow \Aff^1$ given by $a \mapsto \zeta_{a,0}$ is a dense homeomorphic embedding relative to the absolute value topology on $k$ and the weak topology on $\Aff^1$. The type~II points of $\Aff^1$ are dense in $\Aff^1$ for the weak topology.
		
	For $a \in k$ and $r \in \RR_{\geq 0}$, the sets 
	\[
		\DD(a, r)^- = \{ x \in \Aff^1 : |(T-a)(x)| < r \} \quad \text{and} \quad  \DD(a,r) = \{x \in \Aff^1 : |(T-a)(x)| \leq r\}
	\]
are the \textbf{(standard) open Berkovich disk} and the \textbf{(standard) closed Berkovich disk} of radius $r$ about $a$, respectively. 
The weak topology on $\Aff^1$ is generated by sets of the form
	\[
		\DD(a,r)^-
			\ \text{and} \ \Aff^1 \smallsetminus \DD(a, r)
	\]
for $a \in k$ and $r \in \RR_{>0}$. 

\subsubsection{The Strong Topology}	
\label{Sec: Strong Topology}

	The affine line $\Aff^1$ admits a partial ordering $\preceq$ defined by  $x \preceq y$ if and only if $|f(x)| \leq |f(y)|$ for all polynomials $f \in k[T]$. For example, $\zeta_{a, r} \preceq \zeta_{b, s}$ if and only if $D(a, r) \subset D(b, s)$. Given $x, y \in \Aff^1$, the least upper bound with respect to the partial ordering is denoted $x \vee y$. It always exists and is unique. Type~I and type~IV points are the minimal elements with respect to~$\preceq$.
	 	
	Define the \textbf{affine diameter} of the point $\zeta_{a,r}$ to be $\diam(\zeta_{a,r}) = r$. More generally, if $(\zeta_{a_i, r_i})$ is a sequence of seminorms corresponding to a type~IV point $x \in \Aff^1$, define $\diam(x) = \lim r_i$. The limit exists since $(r_i)$ is a decreasing sequence, and $\diam(x) > 0$ (else this sequence corresponds to a type~I point). 
	
	The \textbf{small metric} on $\Aff^1$ is defined by
	\[
		d(x,y) = \left[ \diam(x \vee y) - \diam(x) \right] + \left[ \diam(x \vee y) - \diam(y) \right].
	\]
The topology on $\Aff^1$ induced by $d$ is called the \textbf{strong topology}. 	It is strictly finer than the weak topology.  
	
	Define the \textbf{path-distance metric} $\rho$ on the \textbf{Berkovich hyperbolic space} $\HH = \Aff^1 \smallsetminus k$ via the formula
	\benn
		\ba
		\rho(x,y) 
			 &= \log_{q_k} \frac{\diam(x \vee y)}{\diam(x)} +  \log_{q_k} \frac{\diam(x \vee y)}{\diam(y)}. 
		\ea
	\eenn
The restriction of the strong topology to $\HH$ coincides with the metric topology for $\rho$. The space $\HH$ is complete for this metric, but not locally compact. Note that our choice of normalized base $q_k$ gives $\rho(\zeta_{0, q_k}, \zeta_{0,1}) = 1$. 

	The group $\PGL_2(k)$ acts by isometries for the path-distance metric: $\rho(\sigma(x), \sigma(y)) = \rho(x, y)$ for any $x, y \in \HH$ and $\sigma \in \PGL_2(k)$.

\subsubsection{The Projective Line}

	The Berkovich projective line over $k$, denoted $\Berk = \Berk_k$, is given by gluing two copies of $\Aff^1$ along $\Aff^1 \smallsetminus \{0\}$ via the map $T \mapsto 1 / T$. The weak topology on $\Berk$ is induced by this gluing. We write $\{\infty\} = \Berk \smallsetminus \Aff^1$, and the dense homeomorphic embedding $k \hookrightarrow \Aff^1$ extends to $\PP^1(k) \hookrightarrow \Berk$. We also extend the partial ordering $\preceq$ to $\Berk$ by setting $x \preceq \infty$ for every $x \in \Berk$. For $x \preceq x' \in \Berk$, we define the \textbf{closed segment} $[x,x'] = \{y \in \Berk : x \preceq y \preceq x' \}$, and extend this notion to arbitrary pairs $x, x' \in \Berk$ by $[x,x'] = [x, x \vee x'] \cup [x', x \vee x']$. Open and half-open segments can be defined similarly.
	
	The group $\PGL_2(k)$ acts on $\PP^1(k)$, and this action extends functorially to $\Berk$. Moreover, the action preserves the type of a point in $\Berk$, and it is transitive on the set of type~I and type~II points. The image of an open disk $\DD(a,r)^- \subset \Aff^1$ under the action of an element of $\PGL_2(k)$ will be called an \textbf{open Berkovich disk} (and similarly for a \textbf{closed Berkovich disk}). The weak topology on $\Berk$ is generated by sets of the form $\DD(a,r)^-$ and $\Berk \smallsetminus \DD(a, r)$ 
for $a \in k$ and $r \in \RR_{> 0}$. The space $\Berk$ is compact, Hausdorff, and uniquely path-connected for the weak topology.

	We close with the following important property of the strong and weak topologies on $\Berk$:
	
\begin{prop}[{\cite[Lem.~B.18]{Baker-Rumely_BerkBook_2010}}]
\label{Prop: Strong vs. Weak}
	Let $X \subset \Berk$ be a subset. Then $X$ is connected for the weak topology on $\Berk$ if and only if it is connected for the strong topology on $\Berk$. 
\end{prop}
	
	Consequently, we may speak of the connected components of a subset $X \subset \Berk$ without reference to the topology. 
	
\subsubsection{Tangent vectors}
	
	Let $x \in \Berk$ be a point. Write $T_x$ for the set of connected components of $\Berk \smallsetminus \{x\}$; an element $\vec{v} \in T_x$ will be called a \textbf{tangent vector} at $x$. If we wish to view a connected component $\vec{v} \in T_x$ as a subset of $\Berk$, then we will write it as $\BB_x(\vec{v})^-$. Observe that the weak topology on $\Berk$ is generated by the sets $\BB_x(\vec{v})^-$ as $x$ varies through $\Berk$ and $\vec{v}$ varies through $T_x$. 
	
	The cardinality of the set $T_x$ depends only on the type of the point $x$:
	\begin{enumerate}
		\item Type~I. $T_x$ consists of a single tangent vector.
		\item Type~II. $T_x$ is in 1-to-1 correspondence with elements of $\PP^1(\tilde k)$. 
		\item Type~III. $T_x$ consists of two tangent vectors.
		\item Type~IV. $T_x$ consists of a single tangent vector.
	\end{enumerate} 
In the case of a type~II point $x$, the correspondence between $T_x$ and $\PP^1( \tilde k)$ is non-canonical except when $x = \zeta_{0,1}$. The correspondence $\PP^1(\tilde k) \simarrow T_{\zeta_{0,1}}$ is given by $a \mapsto \vec{a}$, where $\vec{a}$ is the connected component of $\Berk \smallsetminus \{\zeta_{0,1}\}$ all of whose classical points map to $a$ under the canonical reduction map $\PP^1(k) \to \PP^1(\tilde k)$.

\subsection{Rational Functions}

\subsubsection{Generalities}
\label{Sec: General Rational Functions}

	Let $L$ be an algebraically closed field, and let $\varphi \in L(z)$ be a nonconstant rational function. Choose polynomials $f, g \in L[z]$ with no common root such that $\varphi = f / g$. Write $\deg(\varphi) = \max\{ \deg(f), \deg(g)\}$. 
	
	Suppose $x \in \PP^1(L)$ and set $y = \varphi(x)$. Select $\sigma_1, \sigma_2 \in \PGL_2(L)$ such that $\sigma_1(0) = x$ and $\sigma_2(y) = 0$, and define $\psi = \sigma_2 \circ \varphi \circ \sigma_1$. The \textbf{multiplicity} of $\varphi$ at $x$ is defined to be the integer $m_\varphi(x) = \ord_{z = 0} \psi(z)$. Evidently $1 \leq m_\varphi(x) \leq \deg(\varphi)$. The \textbf{weight} of $\varphi$ at $x$ is defined as $w_\varphi(x) = \ord_{z = 0} \psi'(z)$. If $\varphi'(z) \equiv 0$, we set $w_\varphi(x) = +\infty$. The weight and multiplicity at $x$ are  independent of the choice of $\sigma_1$ and $\sigma_2$. 
	
	If $\varphi(x) = y$ with $x, y \neq \infty$, then one verifies that
		\[
			m_\varphi(x) = \ord_{z = x} \left(\varphi(z) - y\right) \qquad 
			w_\varphi(x) = \ord_{z = x} \left(\varphi'(z) \right).
		\] 
As an immediate consequence, we obtain the following formula for each $y \in \PP^1(L)$:
	\[
		\sum_{\substack{x \in \PP^1(L) \\ \varphi(x) = y}} m_\varphi(x) = \deg(\varphi).
	\]

\begin{remark}
	In some of the literature, the multiplicity $m_\varphi(x)$ is referred to as the ``ramification index'' or as the ``local degree.''  The weight $w_\varphi(x)$ is a non-standard terminology special to this paper; it is referred to as the ``multiplicity'' of a critical point in most of the literature. As our focus is on certain multiplicity functions, we have chosen an alternative terminology to avoid confusion.
\end{remark} 
	
	Let $p$ be the characteristic of $L$. The weight and multiplicity of a point are related by
	\[
		w_\varphi(x)   \begin{cases} = m_\varphi(x) - 1 & \text{if $p \nmid m_\varphi(x)$} \\
						 > m_\varphi(x)-1 & \text{if $p \mid m_\varphi(x)$}.
					\end{cases}
	\]
We say that $\varphi$ is \textbf{ramified} (resp. \textbf{unramified}) at $x$ if $m_\varphi(x) > 1$ (resp. $m_\varphi(x) = 1$).  If $p \mid m_\varphi(x)$, we say that $\varphi$ is \textbf{wildly ramified} at $x$; otherwise $\varphi$ is \textbf{tamely ramified} at $x$. 	A point $x$ with positive weight is called a \textbf{critical point} of $\varphi$; the above relations between weights and multiplicities show that $\varphi$ is ramified at $x$ if and only if $x$ is a critical point. We write $\Crit(\varphi)$ for the set of critical points of $\varphi$. 

	If $L$ has characteristic $p > 0$, a rational function $\varphi \in L(z)$ is called \textbf{inseparable} if $\varphi(z) = \psi(z^p)$ for some  rational function $\psi$. Otherwise $\varphi$ is said to be \textbf{separable}. (Equivalently, $\varphi$ is separable if and only if the extension of fields $L(z) / L(\varphi(z))$ is separable.)

	With this notation, the Hurwitz Formula may be written in the following way:

\begin{prop}[Hurwitz Formula]
\label{Prop: Hurwitz}
	Let $\varphi \in L(z)$ be a nonconstant rational function. The collection of weights for $\varphi$ are related by
	\[
		\sum_{x \in \PP^1(L)} w_\varphi(x) = 
			\begin{cases}
				2\deg(\varphi) - 2 & \text{if $\varphi$ is separable} \\
				+\infty & \text{if $\varphi$ is inseparable.}
			\end{cases}
	\]
In particular, a nonconstant separable rational function has at most $2\deg(\varphi) - 2$ distinct critical points. 
\end{prop}

	For a nonconstant rational function $\varphi \in L(z)$, choose polynomials $f, g$ with no common root such that $\varphi = f / g$. This choice is unique up to a common nonzero factor in $L$. The \textbf{Wronskian} of $\varphi = f / g \in L(z)$ is defined to be
		\[
			\wronsk_\varphi = f'g - fg' \in L[z].
		\]
It is a polynomial of degree at most $2\deg(\varphi)-2$ whose roots are precisely the affine critical points of $\varphi$. (If one wants to recover all critical points, then one should work with the homogeneous Wronskian $Y^{2\deg(\varphi)-2}\wronsk_\varphi(X/Y) \in L[X,Y]$.) The Wronskian depends on the choice of representation $f / g$, although we suppress this from the notation. Note that the Hurwitz Formula may be proved by counting roots of the Wronskian with appropriate weights.

	To close this section, we derive an explicit formula for the Wronskian $\wronsk_\varphi$ in terms of the coefficients of $f$ and $g$. Let $d = \deg(\varphi)$ and write 
	\[
		f(z) = a_dz^d + a_{d-1}z^{d-1} + \cdots + a_0, \qquad 
		g(z) = b_dz^d + b_{d-1}z^{d-1} + \cdots + b_0,
	\]
for some coefficients $a_i, b_j \in L$. Let us make the convention that $a_i = b_i = 0$ if $i < 0$ or $i > d$. Then the Wronskian of $\varphi = f / g$ is given by 
	\benn
		\wronsk_\varphi(z) = f'(z)g(z) - f(z)g'(z) = \sum_{i \geq 0} \sum_{j \geq 0} (ia_i b_j - j a_i b_j) z^{i+j-1}.
	\eenn
Making the change of variable $j \mapsto j-i+1$ gives 
	\be
	\label{Eq: Explicit Wronskian}
		\wronsk_\varphi(z) 
			= \sum_{j \geq 0} \left\{ \sum_{i \geq 0} (2i - j - 1) a_i b_{j + 1 - i} \right\} z^j.
	\ee


\subsubsection{Rational functions over non-Archimedean fields}
\label{Sec: Basics non-Archimedean}

	A rational function $\varphi \in k(z)$, viewed as an endomorphism of $\PP^1(k)$, extends functorially to an endomorphism of $\Berk$. By abuse of notation, we denote the extension by $\varphi$ as well. The map $\varphi : \Berk \to \Berk$ is continuous for both the weak and strong topologies. 
	
	Intuitively, the extension $\varphi : \Berk \to \Berk$ reflects the mapping properties of open disks in $\PP^1(k)$. More precisely, if $\varphi$ is nonconstant, we can describe the extension of $\varphi$ to type~II points in the following concrete fashion. Let $S \subset k^\circ$ be a complete collection of coset representatives for $\tilde k = k^\circ / k^{\circ \circ}$. The closed disk $D(0,1)$ is a disjoint union of open disks $D(b,1)^-$ as $b$ varies through $S$. For all but finitely many $b \in S$, the image $\varphi(D(b,1)^-)$ is an open disk $D(\varphi(b), s)^-$ for some $s \in |k^\times|$.  For any such choice of $b$, we have $\varphi(\zeta_{0,1}) = \zeta_{\varphi(b),s}$. For an arbitrary type~II point $\zeta_{a,r}$, choose $\sigma \in \PGL_2(k)$ so that $\sigma(\zeta_{0,1}) = \zeta_{a,r}$, and apply the preceding discussion to the rational function $\varphi \circ \sigma$.  

	Let $\varphi \in k(z)$ be a nonconstant rational function. We may write $\varphi = f / g$ for polynomials $f, g \in k[z]$ with no common root. If $f, g \in k^\circ[z]$ and if the maximum absolute value of the coefficients of $f$ and $g$ is~1, then we say $\varphi$ is  \textbf{normalized}.  

	Given $\varphi \in k(z)$, we may always choose polynomials $f, g \in k^{\circ}[z]$ so that $\varphi = f / g$ is normalized. (It can be accomplished by dividing the numerator and denominator of an arbitrary representation by a judicious choice of nonzero element of $k$.) This choice of $f$ and $g$ is unique up to simultaneous multiplication by an element in $k$ with absolute value~1. Write $\tilde f$ and $\tilde g$ for the images of $f$ and $g$ in $k^\circ[z] / k^{\circ \circ}[z]$, respectively. The \textbf{reduction} of $\varphi$ is given by 
	\[
		\widetilde{\varphi}(z) = \begin{cases}
				 \tilde f / \tilde g & \text{if } g \not\in k^{\circ \circ}[z] \\
				 \infty & \text{if } g \in k^{\circ \circ}[z].
			\end{cases}
	\]
The degree of $\widetilde{\varphi}$ is independent of the choice of normalized representation $\varphi = f / g$. (By convention, we set $\deg(\infty) =  0$.) We say that $\varphi$ has \textbf{constant reduction} (resp. \textbf{nonconstant reduction}) if the degree of $\widetilde{\varphi}$ is zero (resp. positive). 
 
\begin{prop}
	Let $\varphi \in k(z)$ be a nonconstant rational function, and write $\varphi = f / g$ in normalized form. Then $\varphi$ has nonconstant reduction if and only if $\varphi(\zeta_{0,1}) = \zeta_{0,1}$. 
\end{prop}

\begin{proof}
	This is essentially Lemma~2.17 of \cite{Baker-Rumely_BerkBook_2010}. As the point at infinity plays a distinguished role in much of their theory, they do not treat the case in which $\varphi$ has constant reduction with value $\infty$. This issue can be remedied by replacing $\varphi(z)$ with $1 / \varphi(1/z)$. 
\end{proof}

	Let $x$ be a point of $\Berk$ and $\vec{v}$ a tangent direction at $x$. Then for every $y \in \BB_x(\vec{v})^-$ sufficiently close to $x$, the image segment $\varphi((x,y))$ does not contain $\varphi(x)$, and hence it lies entirely in a single connected component of $\Berk \smallsetminus \{\varphi(x)\}$. In this way, $\varphi$ determines a surjective map $\varphi_* : T_x \to T_{\varphi(x)}$ \cite[Cor.~9.20]{Baker-Rumely_BerkBook_2010}. We have already seen that $T_{\zeta_{0,1}}$ is canonically identified with $\PP^1(\tilde k)$. If $\varphi(\zeta_{0,1}) = \zeta_{0,1}$, then under this identification we have $\widetilde{\varphi} = \varphi_*$.


\section{Multiplicity Functions}
\label{Sec: Multiplicities}

\subsection{Extending $m_\varphi$ to $\Berk$}

	Here we describe an extension of the multiplicity function $m_\varphi$ on $\PP^1(k)$ to the Berkovich projective line $\Berk$, where $k$ is a non-Archimedean field. There are a number of equivalent ways to do this; see  \cite[\S9.1]{Baker-Rumely_BerkBook_2010}, \cite[\S6.3.1]{Berkovich_Etale_1993}, and \cite[\S2.2]{Favre_Rivera-Letelier_Ergodic_2010}. The definition is relatively unimportant for our purposes in this paper (although we give one for completeness); instead, we rely on various characterizations and properties of the multiplicity function to be recalled below. 
	
	The most direct definition of the multiplicity function is as follows. Let $k$ be a non-Archimedean field, let $\varphi \in k(z)$ be a nonconstant rational function, and let $\OO_{\Berk}$ be  the analytic structure sheaf on $\Berk$. Then $\varphi_*\OO_{\Berk}$ is a locally free $\OO_{\Berk}$ module. The \textbf{multiplicity} of $\varphi$ at $x \in \Berk$ is defined as 
	\[
		m_\varphi(x) = \mathrm{rk}_{\OO_{\Berk, y}} (\varphi_*\OO_{\Berk})_x
			= \mathrm{rk}_{\OO_{\Berk, y}} \OO_{\Berk, x}, \qquad y = \varphi(x).
	\]

	More intuitively, we have the following topological characterization that appears in the work of Rivera-Letelier. 
It will be the first instance of many in which we want to count a set of points ``with multiplicities.'' To be precise, if $X \subset \Berk$ is a set, then to count $X$ \textbf{with multiplicities} means to compute the quantity
		$\#X = \sum_{x \in X} m_\varphi(x)$. 

\begin{prop}[{\cite[Cor.~9.17]{Baker-Rumely_BerkBook_2010}}]
\label{Prop: Top Characterization}
	For each $x \in \Berk$ and for each sufficiently small $\varphi$-saturated neighborhood $U$ of $x$ (i.e., $U$ is a connected component of $\varphi^{-1}(\varphi(U))$), the multiplicity $m_\varphi(x)$ is equal to $\#U \cap \varphi^{-1}(\{b\})$ for each $b \in \varphi(U) \cap \PP^1(k)$. 
\end{prop}

	Intuitively, this says that each classical point near $\varphi(x)$ has $m_\varphi(x)$ pre-images when counted with multiplicities. More generally, it is true that if $x \in \Berk$ has multiplicity $m = m_\varphi(x)$, then the map $\varphi$ is locally $m$-to-1 in a neighborhood of $x$, provided that we count with multiplicities. The function $m_\varphi: \Berk \to \{1, \ldots, \deg(\varphi)\}$ is sometimes called the ``local degree function'' for this reason. 
	
\begin{define} 
The \textbf{(Berkovich) ramification locus} of a nonconstant rational function $\varphi \in k(z)$ is the set
	\[
		\Ram_\varphi = \{x \in \Berk : m_\varphi(x) > 1 \}.
	\]
\end{define}

\begin{remark}
	We call an arbitrary point $x \in \Ram_\varphi$ a \textbf{ramified point}, while we reserve the term ``critical point'' for the type~I points in $\Ram_\varphi$. (A different convention is used in \cite{Trucco_Tame_Polynomials_2012}.)
\end{remark}

	A rational function $\varphi$ of degree~1 is an automorphism. Hence $m_\varphi \equiv1$ on $\Berk$, so that $\Ram_\varphi$ is empty. A rational function $\varphi$ with $\deg(\varphi) \geq 2$ has a critical point --- i.e., a classical point of multiplicity at least~2 --- and so $\Ram_\varphi$ is nonempty. 

\begin{prop}[{\cite[Prop.~9.28]{Baker-Rumely_BerkBook_2010}}]
\label{Prop: Multiplicity Properties}
	Let $\varphi \in k(z)$ be a nonconstant rational function. The multiplicity function $m_\varphi: \Berk \to \{1, \ldots, \deg(\varphi)\}$ enjoys the following properties. 
\begin{enumerate}
	\item\label{Item: Semicontinuity} 
		$m_\varphi$ is upper semicontinuous with respect to the weak topology. 
		That is, the set \newline $\{x \in \Berk : m_\varphi(x) \geq i \}$ is weakly closed 
		in $\Berk$ for each $i = 1, 2, \ldots, \deg(\varphi)$. 

	\item\label{Item: Injective} The map $\varphi: \Berk \to \Berk$ is locally injective at $a$ with respect to the
		weak topology if $m_\varphi(a) = 1$. The converse holds if $\varphi$ is separable. 
	
	\item\label{Item: Multiplicativity} 
		If $\psi(z)$ is another nonconstant rational function, then 
		\[
			m_{\psi \circ \varphi}(x) = m_\psi(\varphi(x))\cdot m_\varphi(x)  \quad \text{for all $x \in \Berk$.}
		\]
\end{enumerate}
\end{prop}

\begin{remark}
	Statements~\eqref{Item: Semicontinuity} and~\eqref{Item: Injective} are also true for the strong topology. 
\end{remark}

\begin{remark}
	Part~\eqref{Item: Injective} of the proposition is proved in \cite{Baker-Rumely_BerkBook_2010} under the hypothesis that the characteristic of $k$ is zero, but their proof applies \textit{mutatis mutandis} if $\varphi$ is separable. See also \cite[\S2]{Favre_Rivera-Letelier_Ergodic_2010}.
\end{remark}

\begin{cor}
\label{Cor: Coord Change}
	Let $\varphi \in k(z)$ be a nonconstant rational function, and let $\sigma_1, \sigma_2 \in \mathrm{PGL}_2(k)$. Set $\psi = \sigma_2 \circ \varphi \circ \sigma_1$. Then $\Ram_\psi = \sigma_1^{-1}(\Ram_\varphi)$. 
\end{cor}

\begin{proof}
	This result is an immediate consequence of part~\eqref{Item: Multiplicativity} of the proposition and the fact that automorphisms are unramified:
		\[
			m_\psi(x) = m_{\sigma_2}\left(\varphi(\sigma_1(x))\right) \cdot m_{\varphi}\left(\sigma_1(x)\right)
			\cdot m_{\sigma_1}(x) = m_\varphi \left(\sigma_1(x)\right), \qquad x \in \Berk. \qedhere
		\]
\end{proof}

	The fact that $\Berk$ is a tree implies that a rational function is injective on each connected component of the complement of the ramification locus. 

\begin{cor}
	Let $\varphi \in k(z)$ be a nonconstant rational function, and let $U \subset \Berk$ be a connected weak open subset. If $\varphi|_U$ is not injective, then $U$ contains a ramified point.
\end{cor}

\begin{proof}
	Let $x,y$ be arbitrary distinct points of $U$. The segment $[x,y]$ is contained in $U$ by connectedness.  If $U$ does not contain a ramified point, then $\varphi$ is locally injective at every point of $[x,y]$ (Proposition~\ref{Prop: Multiplicity Properties}). In particular, the image path $[x,y] \to \varphi([x,y])$ cannot have any backtracking. As $\Berk$ contains no loop, it follows that $\varphi(x) \neq \varphi(y)$, so that $\varphi$ is injective.
\end{proof}


\subsection{The Directional Multiplicity}
\label{Sec: Directional Multiplicity}

	Essentially all of the ideas in this section are due to Rivera-Letelier\cite[\S4]{Rivera-Letelier_Espace_Hyperbolique_2003}, although we will adhere to the notation and terminology of Baker and Rumely\cite[\S9.1]{Baker-Rumely_BerkBook_2010}. 
	
\begin{prop}[{\cite[pp.261--266]{Baker-Rumely_BerkBook_2010}}]
\label{Prop: Directional Definition}
	Let $\varphi \in k(z)$ be a nonconstant rational function, let $x \in \Berk$, and let $\vec{v} \in T_x$. Then there is a positive integer $m$ and a point $x' \in \BB_x(\vec{v})^-$ satisfying the following:
	\begin{enumerate}
		\item\label{Item: Locally constant} $m_\varphi(y) = m$ for all $y \in (x,x')$, and
		\item\label{Item: Local scaling} $\rho(\varphi(x), \varphi(y)) = m \cdot \rho(x,y)$ for all $y \in (x,x')$.
	\end{enumerate}
\end{prop}

	The integer $m$ in the proposition is called the \textbf{directional multiplicity}, and we denote it by $m_\varphi(x,\vec{v})$. Part~\eqref{Item: Locally constant} shows that it satisfies $m_\varphi(x, \vec{v}) \leq \deg(\varphi)$.


	 For the next statement, a \textbf{generalized open Berkovich disk} is a weakly open set of the form $\BB_x(\vec{v})^-$ for some point $x \in \Berk$ and some tangent vector $\vec{v}$ at $x$. Equivalently, a weak open subset is a generalized open Berkovich disk if and only if it has exactly one boundary point. 

\begin{prop}
	\label{Prop: Rivera-Letelier Mapping}
	Let $\varphi \in k(z)$ be a nonconstant rational function. Let $\BB = \BB_x(\vec{v})^-$ be a generalized open Berkovich disk.
Then $\varphi(\BB)$ always contains the generalized open Berkovich disk $\BB' = \BB_{\varphi(x)}(\varphi_*(\vec{v}))^-$, and either $\varphi(\BB) = \BB'$ or $\varphi(\BB) = \Berk$. Set $m = m_\varphi(x, \vec{v})$ for the directional multiplicity.
	\begin{enumerate}
		\item If $\varphi(\BB) = \BB'$, then for each $y \in \BB'$ there are exactly $m$ solutions to 
			$\varphi(z) = y$ in $\BB$ (counted with multiplicities).
			
		\item\label{Item: Surplus Multiplicity} 
			If $\varphi(\BB) = \Berk$, then there there is a unique integer $s > 0$ such that for each $y \in \BB'$, 
			there are $s+m$ solutions to $\varphi(z) = y$ in $\BB$ (counted with multiplicities), and
			for each $y \in \Berk \smallsetminus \BB'$ there are $s$ solutions to 
			$\varphi(z) = y$ in $\BB$ (counted with multiplicities).
	\end{enumerate}
\end{prop}

\begin{proof}
	The proposition seems to have been known to Rivera-Letelier \cite[\S4.1]{Rivera-Letelier_Espace_Hyperbolique_2003}, but it was only stated in the case where $y$ is a classical point. Baker and Rumely give a proof of the full statement except for the case where $\varphi(\BB) =\Berk$ and $y = \varphi(x)$ \cite[Prop.~9.41]{Baker-Rumely_BerkBook_2010}. If $\varphi(\BB) = \Berk$ and $x$ is of type~I or type~IV, then $y = \varphi(x)$ is the unique point in $\Berk \smallsetminus \BB'$. Evidently the desired result holds with $s = \deg(\varphi) - m$.  We will supply the remaining case now using a perturbation argument. 
	
	Suppose that $x$ is of type~II or type~III and $\varphi(\BB) = \Berk$. The result of Baker and Rumely tells us that there is an integer $s > 0$ such that for each $y \in \Berk \smallsetminus \overline{\BB'}$, $y$ has $s$ pre-images inside $\BB$ (counted with multiplicities). We must now extend this statement to the point $y = \varphi(x)$. There exists a segment $I = [x, x']$ with $x' \in \BB$ such that $m_\varphi$ is constant with value $m$ on the interior of $I$ (Proposition~\ref{Prop: Directional Definition}) and such that $\varphi$ is injective on $I$ \cite[Thm.~9.35]{Baker-Rumely_BerkBook_2010}. Fix an ancillary element $y_1 \neq y$ in the complement of $\BB'$. Select an open Berkovich disk $\BB_0 \subsetneq \BB$ of the form $\BB_0 = \BB_{x_0}(\vec{v}_0)^-$ satisfying the following properties:
	\begin{itemize}
		\item $x_0$ lies in the interior of the segment $I$;
		\item $\varphi^{-1}(y) \cap \BB = \varphi^{-1}(y) \cap \BB_0$ and 
			$\varphi^{-1}(y_1) \cap \BB = \varphi^{-1}(y_1) \cap \BB_0$; and
		\item $\varphi(\BB_0) = \Berk$.
	\end{itemize}
Each of these three properties holds for any sufficiently large subdisk of $\BB$:  the second because each element of $\Berk$ has only finitely many pre-images under~$\varphi$, and the third by compactness. 

Write $\BB'_0 = \BB_{\varphi(x_0)}(\varphi_*(\vec{v}_0))^-$.  Then $\BB'_0 \subsetneq \BB'$ because $\varphi$ is injective on $I$, and hence $y, y_1 \in \Berk \smallsetminus \overline{\BB_0'}$. Applying the case already proved by Baker and Rumely to the disks $\BB_0$ and $\BB$ separately, we find that there is a unique integer $s_0 > 0$ such that $y$ and $y_1$ each has $s_0$ pre-images in $\BB_0$, counted with multiplicities. That is, 
	\[
		\# \left(\varphi^{-1}(y) \cap \BB\right) = \# \left(\varphi^{-1}(y) \cap \BB_0\right) = s_0 = \# \left(\varphi^{-1}(y_1) \cap \BB_0\right) =
		\# \left(\varphi^{-1}(y_1) \cap \BB\right) = s.  \qedhere
	\]
\end{proof}

	The next result gives an algebraic relationship between the multiplicity $m_\varphi(x)$ and the directional multiplicities $m_\varphi(x, \vec{v})$ for $\vec{v} \in T_x$. 
	
\begin{prop}[{\cite[Thm.~9.22]{Baker-Rumely_BerkBook_2010}}]
\label{Prop: Expansion Properties}
	Let $\varphi \in k(z)$ be a nonconstant rational function and let $x \in \Berk$.
	\begin{enumerate}
		
		\item(Directional Multiplicity Formula)
			\label{Item: Directional Sum} For each tangent vector $\vec{w}$ at $\varphi(x)$, we have
				\[
					m_\varphi(x) = \sum_{\substack{\vec{v} \in T_x \\ \varphi_*(\vec{v}) = \vec{w}}}
						m_\varphi(x, \vec{v}). 
				\]
		\item
			\label{Item: Tangent Properties} 
			The induced map $\varphi_*: T_x \to T_{\varphi(x)}$ is surjective.			
			If $x$ is of type~I,~III, or~IV, then $m_\varphi(x) = m_\varphi(x, \vec{v})$ for 
			each tangent vector $\vec{v} \in T_x$. 
	\end{enumerate} 
\end{prop}

\begin{cor}
\label{Cor: Closed}
	Let $\varphi \in k(z)$ be a nonconstant rational function. Then $\Ram_\varphi$ is a closed subset of $\Berk$ with no isolated point (for both the weak and strong topologies). 
\end{cor}

\begin{proof}	
	Proposition~\ref{Prop: Multiplicity Properties}\eqref{Item: Semicontinuity} immediately implies that $\Ram_\varphi$ is closed. Proposition~\ref{Prop: Expansion Properties}\eqref{Item: Directional Sum} shows that if $m_\varphi(x) > 1$, then there is a direction $\vec{v} \in T_x$ such that $m_\varphi(x, \vec{v}) > 1$. It follows from Proposition~\ref{Prop: Directional Definition}\eqref{Item: Locally constant} that there exists $x' \in \BB_x(\vec{v})^-$ such that $m_\varphi(y) = m_\varphi(x, \vec{v}) > 1$ for all $y \in (x,x')$. Hence $\Ram_\varphi$ has no isolated point. 
\end{proof}

The following proposition, due to Rivera-Letelier, gives the best technique for determining the value of the multiplicity function at a type~II point.

\begin{prop}[Algebraic Reduction Formula, {\cite[Thm.~9.42]{Baker-Rumely_BerkBook_2010}}]
	Let $\varphi \in k(z)$ be a nonconstant rational function, and let $x \in \Berk$ be a point of type~II. Put $y = \varphi(x)$, choose $\sigma_1, \sigma_2 \in \PGL_2(k)$ such that $\sigma_1(x) = \sigma_2(y) = \zeta_{0,1}$, and set $\psi(z) = \sigma_2 \circ \varphi \circ \sigma_1^{-1}$. Then $\psi$ has nonconstant reduction $\tilde{\psi}$ and 
		\[
			m_\varphi(x) = \deg(\tilde{\psi}).
		\]
For each $a \in \PP^1(\tilde{k})$, if $\vec{v}_a \in T_x$ is the associated tangent direction under the bijection between $T_x$ and $\PP^1(\tilde{k})$ afforded by $(\sigma_1)_*$, we have
		\[
			m_\varphi(x, \vec{v}_a) = m_{\tilde{\psi}}(a).
		\]
\end{prop}

	As an application of the results in this section, we describe the ramification locus for inseparable rational functions.

\begin{prop}
\label{Prop: Frobenius}
	Suppose $k$ has characteristic~$p > 0$ (and hence residue characteristic $p$, in accordance with our conventions). Let $\varphi \in k(z)$ be a nonconstant inseparable rational function. Then $\Ram_\varphi = \Berk$. 
\end{prop}

\begin{proof}	
	We begin by showing that $m_F \equiv p$, where $F \in k(z)$  is the relative Frobenius map defined by $F(z) = z^p$. For a closed disk $D(a, r)$ with rational radius and any $b \in D(a, r)$ with $|a-b| = r$, observe that 
		\[
			\psi(z) = b^{-p}[F(bz + a) - F(a)] = F(z).
		\]
Hence $m_F(\zeta_{a, r}) = m_F(\zeta_{0,1}) = p$ by the Algebraic Reduction Formula. Since $m_F$ takes the same value at any type~II point, we conclude that $m_F \equiv p$ (Proposition~\ref{Prop: Directional Definition}\eqref{Item: Locally constant}).

	Now we may  factor $\varphi$ uniquely as $\varphi  =  \psi \circ F^\ell$, where $\psi \in k(z)$ is separable, $\ell \geq 1$, and $F^{\ell} = F \circ \cdots \circ F$ is the $\ell$-fold iterate of $F$. 
By Proposition~\ref{Prop: Multiplicity Properties}\eqref{Item: Multiplicativity}, we see that 
		\begin{align*}
			m_\varphi(x) &= 
			m_\psi\left(F^{\ell}(x)\right) \cdot m_F\left(F^{\ell-1}(x)\right) 
			\cdot m_F\left(F^{\ell-2}(x)\right) \cdots m_F(x) \\
			&= p^\ell \cdot m_\psi\left(F^{\ell}(x)\right) \geq p^\ell.
		\end{align*}
As $x$ is arbitrary, $\Ram_\varphi = \Berk$. 
\end{proof}


\subsection{The Surplus Multiplicity}
\label{Sec: Surplus Multiplicity}

With the notation in Proposition~\ref{Prop: Rivera-Letelier Mapping}, we define the \textbf{surplus multiplicity} $s_\varphi(\BB)$ to be zero if $\varphi(\BB)$ is a generalized open Berkovich disk, and to be $s_\varphi(\BB) = s$ if $\varphi(\BB) = \Berk$. As $\BB = \BB_x(\vec{v})^-$, we will also write $s_\varphi(x, \vec{v}) = s_\varphi(\BB)$.	 The intuition behind the terminology ``surplus multiplicity'' is that for $y \in \BB'$, there are always at least $m$ solutions to $\varphi(\zeta) = y$ with $\zeta \in \BB$, and there are $s_\varphi(\BB)$ ``extra'' solutions depending on the nature of $\varphi$ and $\BB$. 

	The surplus multiplicity gives a lower bound for the number of pre-images of a given point inside certain open Berkovich disks. This fact --- which follows immediately from Proposition~\ref{Prop: Rivera-Letelier Mapping} --- is extremely important for bounding the number of connected components of $\Ram_\varphi$. 

\begin{cor}
\label{Rem: Lower Bound}
	Let $\varphi \in k(z)$ be a nonconstant rational function, and let $\BB$ be a generalized open Berkovich disk. For each $y \in \Berk$, 
	\[
		\#\{\zeta \in \BB : \varphi(\zeta) = y\} \geq s_\varphi(\BB).
	\] 
\end{cor}

	The surplus multiplicity of a disk is closely tied to the number of critical points contained within it. The following result  is the key to bounding the number of connected components of the ramification locus. 
	
\begin{prop}
\label{Prop: Critical surplus}
	Let $\varphi \in k(z)$ be a nonconstant rational function. Suppose $x \in \Berk$ is a type~II point and $\vec{v} \in T_x$ is a tangent direction such that $p \nmid m_\varphi(x, \vec{v}) $. Then we have the equality
	\[
		\sum_{c \in \Crit(\varphi) \cap \BB_x(\vec{v})^-} w_\varphi(c) 
			= 2 s_\varphi(x, \vec{v}) + m_\varphi(x, \vec{v}) - 1.
	\]
\end{prop}

Before starting the proof, we give an alternate description of the surplus multiplicity at the Gauss point. Let $\varphi = f / g$ be normalized. The surplus multiplicity is invariant under postcomposition by an element of $\PGL_2(k)$, so it suffices to assume that $\varphi(\zeta_{0,1}) = \zeta_{0,1}$, in which case $\varphi$ has nonconstant reduction. In particular, this means that each of $f$ and $g$ has a coefficient with absolute value~1. 
	
	Write $F(X,Y) = Y^{\deg(\varphi)}f(X/Y)$ and $G(X,Y) = Y^{\deg(\varphi)}g(X/Y)$ for the homogenizations of $f$ and $g$. Write $\tilde{F}$ and $\tilde{G}$ for the reductions of $F$ and $G$, respectively; these reductions are nonzero since $f$ and $g$ each has a coefficient with absolute value~1. Let $H = \gcd(\tilde{F}, \tilde{G}) \in \tilde{k}[X,Y]$; it exists since $\tilde{F}$ and $\tilde{G}$ are homogeneous, and it is unique up to multiplication by a nonzero element of the residue field. 
	
	Now let $a \in \PP^1(\tilde{k})$, and write $\BB = \BB_{\zeta_{0,1}}(\vec{a})^-$ for the corresponding open Berkovich disk. We claim that the surplus multiplicity of $\BB$ is equal to the multiplicity of $a$ as a root of $H$. To see it, change coordinates on the source and target by an element of $\PGL_2(k^\circ)$ so that $\vec{a} = \varphi_*(\vec{a}) = \vec{0}$. The induced map $T_{\zeta_{0,1}} \to T_{\zeta_{0,1}}$ on sets of tangent vectors is given in homogeneous coordinates by
	\[
		\varphi_* = \tilde{\varphi} = \left( \frac{\tilde{F}}{H} : \frac{\tilde{G}}{H} \right).
	\]
Since $\varphi_*$ maps $\vec 0$ to $\vec 0$ with multiplicity $m = m_\varphi(\zeta_{0,1}, \vec{0})$, we see that $X^m \mid\mid \tilde{F} / H$. Let $S \geq 0$ be defined by $X^S \mid \mid H$. It follows that $X^{m+S}$ evenly divides $\tilde{F}$, or equivalently that $F$ has $m + S$ zeros in the disk $D(0,1)^-$ (counted with multiplicity). In fact, this same conclusion holds with zero replaced by any $y \in D(0,1)^-$, which shows that $S = s_\varphi(\zeta_{0,1}, \vec{0})$ is the surplus multiplicity of the disk $\DD(0,1)^-$. We summarize this conclusion as

\begin{lem}
\label{Lem: Surplus Multiplicity}
	Let $\varphi = f/g \in k(z)$ be a nonconstant normalized rational function with nonconstant reduction. Set $F(X,Y) = Y^{\deg(\varphi)}f(X/Y)$ and $G(X,Y) = Y^{\deg(\varphi)}g(X/Y)$ for the homogenizations of $f$ and $g$, respectively, and let $H = \gcd(\tilde{F}, \tilde{G})$ be a greatest common divisor of their reductions. For each $a \in \PP^1(\tilde{k})$, the surplus multiplicity of the disk $\BB_{\zeta_{0,1}}(\vec a)^-$ is equal to the multiplicity of $a$ as a root of~$H$. 
\end{lem}

	If $x \in \Berk$ is a type~II point and $\varphi \in k(z)$ is any nonconstant rational function, then an immediate consequence of this characterization of the surplus multiplicity and the Algebraic Reduction Formula is the following:
	\be
	\label{Eq: Surplus Balance Formula}
		m_\varphi(x) + \sum_{\vec{v} \in T_x}s_\varphi(x, \vec{v}) = \deg(\varphi).
	\ee
And while we will not need it in what follows, this formula actually holds at any $x \in \Berk$. The proof is trivial for points of type~I or type~IV since there is only one tangent direction to consider, and one can use Corollary~\ref{Cor: Compatible Multiplicities} below to reduce the type~III case to the type~II case.

\begin{proof}[Proof of Proposition~\ref{Prop: Critical surplus}]
	Change coordinates on the source and target so that $x = \varphi(x) = \zeta_{0,1}$ and $\BB_x(\vec{v})^- = \DD(0,1)^-$.	Note that $\varphi$ must be separable, else its reduction $\widetilde{\varphi}$ will be inseparable, so that $m_\varphi(x, \vec{v}) \geq p$ by the Algebraic Reduction Formula. In particular, $\varphi$ has only finitely many critical points, so there are only finitely many connected components of $\Berk \smallsetminus \{\zeta_{0,1}\}$ that contain one. After a further change of coordinate on the source if necessary, we may assume that no critical point lies in the open Berkovich disk $\BB_{\zeta_{0,1}}(\vec{\infty})^-$; equivalently, each critical point has absolute value at most~1. 
	
	We may suppose $\varphi = f / g$ is normalized, and set $h = \gcd(\tilde{f}, \tilde{g})$ with $h$ monic. Write $\tilde{f} = h f_1$ and $\tilde{g} = h g_1$ for some $f_1, g_1 \in \tilde{k}[z]$. We see that $f_1$ vanishes to order $m = m_\varphi(\zeta_{0,1}, \vec{0})$ at the origin, and $g_1(0) \neq 0$. As $p \nmid m$, we have
	\[
		\ord_{z = 0} \wronsk_{\widetilde{\varphi}} 
			= \ord_{z=0}( f_1'  g_1 - f_1  g_1') = m - 1.	
	\]

	Since $\varphi = f/g$ is normalized, we see that $\wronsk_\varphi \in k^\circ[z]$, and we may compute
	\begin{align*}
		\widetilde{\wronsk}_\varphi &= \tilde{f}'  \tilde{g} - \tilde{f}  \tilde{g}' \\
			&= (h' f_1 + h f_1')hg_1 - hf_1(hg_1' + h' g_1) \\
			&=  \wronsk_{\widetilde{\varphi}} \cdot h^2.
	\end{align*}
With our choice of coordinates, all of the roots of $\wronsk_\varphi$ have absolute value at most~1, and hence 
	\[
		\sum_{c \in \Crit(\varphi) \cap \DD(0,1)^-} w_\varphi(c) = \ord_{z=0} \widetilde{\wronsk}_\varphi 
			= 2s_\varphi(\zeta_{0,1}, \vec{0}) + m_\varphi(\zeta_{0,1}, \vec{0}) - 1,
	\]
where $\ord_{z=0}(h) = s_\varphi(\zeta_{0,1}, \vec{0})$ follows upon dehomogenizing Lemma~\ref{Lem: Surplus Multiplicity}.
\end{proof}

	We now give another useful description of the surplus multiplicity of an open Berkovich disk $\BB$ as a sum of ``jumps'' in the multiplicity function inside $\BB$. 
	
\begin{prop}
\label{Prop: Jumping surplus}
	Let $\varphi \in k(z)$ be a nonconstant rational function, and let $\BB$ be a generalized open Berkovich disk with boundary point $\zeta$. Then
	\[
		s_\varphi(\BB) = \sum_{y \in \BB} 
			\max\left\{ m_\varphi(y) - m_\varphi(y, \vec{v}_\zeta), \; 0 \right\},
	\]
where $\vec{v}_\zeta$ is the unique tangent vector at $y$ containing $\zeta$. 
\end{prop}

\begin{remark}
	Since $m_\varphi(y, \vec{v}) = m_\varphi(x)$ for all $x \in (\zeta, y)$ sufficiently close to $y$, we can think of $\max\left\{ m_\varphi(y) - m_\varphi(y, \vec{v}_\zeta), \; 0 \right\}$ as the ``jump'' in multiplicity at $y$ along a path emanating from $\zeta$. 
\end{remark}

	Since the surplus multiplicity of a disk $\BB$ is positive if and only if $\varphi(\BB) = \Berk$, we obtain the following corollary. It appeared previously as \cite[Thm.~9.42]{Baker-Rumely_BerkBook_2010}. 

\begin{cor}
\label{Cor: Nonincreasing}
	Let $\varphi \in k(z)$ be a nonconstant rational function, and let $\BB$ be a generalized open Berkovich disk with boundary point $\zeta$. Then $\varphi(\BB)$ is a generalized open Berkovich disk if and only if for each $c \in \BB$, the multiplicity function $m_\varphi$ is nonincreasing on the directed segment $[\zeta, c]$.
\end{cor}

\begin{proof}[Proof of Proposition~\ref{Prop: Jumping surplus}]
	For the purpose of this proof, let us make two ad hoc definitions. We will say that $y \in \BB$ is a \textbf{jumping point} if $m_\varphi(y) > m_\varphi(y, \vec{v}_\zeta)$; these are precisely the points that contribute to the sum in the proposition. Note that a jumping point is necessarily of type~II (Proposition~\ref{Prop: Expansion Properties}\eqref{Item: Tangent Properties}). We say that a jumping point $y$ is \textbf{visible from $\zeta$} if it is the unique jumping point on the path $(\zeta, y]$. 
	
	Let $y \in \BB$ be a jumping point that is visible from $\zeta$. We claim that $\varphi$ is injective on the segment $(\zeta, y)$. Otherwise $\varphi$ would have to backtrack on this segment, which would imply there is $x \in (\zeta, y)$ and a tangent vector $\vec{u}$ at $x$ such that $\varphi_*(\vec{v}_\zeta) = \varphi_*(\vec{u})$. But then 
	\[
		m_\varphi(x) \geq m_\varphi(x, \vec{v}_\zeta) + m_\varphi(x, \vec{u}) > m_\varphi(x, \vec{v}_\zeta),
	\]
so that $x$ is a jumping point. This contradicts the visibility of $y$. 

	 Define a closed Berkovich disk $D = \Berk \smallsetminus \BB_y(\vec{v}_\zeta)^-$. Write $\vec{w} = \varphi_*(\vec{v}_\zeta) \in T_{\varphi(y)}$. Let $\vec{v}_1, \ldots, \vec{v}_n \in T_y$ be the distinct tangent vectors at $y$ that are distinct from $\vec{v}_\zeta$ and that satisfy $\varphi_*(\vec{v}_i) = \vec{w}$. Let $\vec{v}_{n+1}, \ldots, \vec{v}_N$ be the remaining tangent vectors distinct from $\vec{v}_\zeta$ satisfying $s_\varphi(y, \vec{v}_i) > 0$. Set $\zeta' = \varphi(\zeta)$. By the Directional Multiplicity Formula and Proposition~\ref{Prop: Rivera-Letelier Mapping}, we find that 
	 \be
	 \label{Eq: Surplus recursion}
	 	\ba
	 	\#\left(\varphi^{-1}(\zeta') \cap D\right) &= \sum_{i=1}^n m_\varphi(y, \vec{v}_i) + \sum_{i=1}^N s_\varphi(y, \vec{v}_i) \\
			&= m_\varphi(y) - m_\varphi(y, \vec{v}_\zeta) 
				+ \sum_{\vec{v} \in T_y \smallsetminus \{\vec{v}_\zeta\}} s_\varphi(y, \vec{v}). 
	 	\ea
	\ee
Here we are counting pre-images of $\zeta'$ with multiplicities. Since $y$ is a jumping point, we conclude that $D$ contains a pre-image of $\zeta'$. It follows that there can be only finitely many jumping points visible from $\zeta$. 

	We now complete the proof by induction on the surplus multiplicity $s = s_\varphi(\BB)$. If $s = 0$, then $\zeta' \not\in \varphi(\BB)$. A visible jumping point $y \in \BB$ would enable us to apply \eqref{Eq: Surplus recursion}; the right side would be strictly positive, which would force the existence of a pre-image of $\zeta'$ in $\BB$. Hence there can be no jumping point when $s = 0$ and the proposition is proved. Suppose now that the proposition holds for all open Berkovich disks with surplus multiplicity at most $s-1$ for some $s \geq 1$, and let $\BB$ be a disk with surplus multiplicity $s$. Let $y_1, \ldots, y_\ell$ be the jumping points in $\BB$ that are visible from $\zeta$. For $i = 1, \ldots, \ell$, let $D_i = \Berk \smallsetminus \BB_{y_i}(\vec{v}_\zeta)^-$ be the closed Berkovich disk with boundary point $y_i$ that does not contain $\zeta$. The argument given in the first paragraph shows that any path $(\zeta, y]$ that contains a pre-image of $\zeta'$ must necessarily contain a jumping point. Hence, \eqref{Eq: Surplus recursion} gives
	\benn
		\ba
		\#\left(\varphi^{-1}(\zeta') \cap \BB \right) &= \sum_{i = 1}^\ell \#\left(\varphi^{-1}(\zeta') \cap D_i \right) \\
			&= \sum_{i = 1}^\ell \left( m_\varphi(y_i) - m_\varphi(y_i,\vec{v}_\zeta) \right)+ \sum_{i=1}^\ell \sum_{\vec{v} \in T_{y_i} \smallsetminus \{\vec{v}_\zeta\}} s_\varphi(y_i, \vec{v}) .
		\ea
	\eenn
But now observe that each of the surplus multiplicities $s_\varphi(y_i, \vec{v})$ is necessarily smaller than $s$, so that we may apply the inductive hypothesis to each of the open disks $\BB_{y_i}(\vec{v})^-$. We conclude that
	\[
		\#\left(\varphi^{-1}(\zeta') \cap \BB \right) = \sum_{y \in \BB} \max\left\{ m_\varphi(y) - m_\varphi(y, \vec{v}_\zeta), \; 0 \right\}.
	\]
Proposition~\ref{Prop: Rivera-Letelier Mapping} shows that the pre-image count on the left hand side is precisely $s_\varphi(\BB)$. 
\end{proof}


\section{Extension of Scalars}
\label{Sec: Base Change}

	Let $K / k$ be an extension of algebraically closed and complete non-Archimedean fields, where as usual we assume the absolute value on $k$ is nontrivial (and hence also on $K$). To distinguish between objects defined over $k$ and those over $K$, we will decorate our notation with subscripts. For example, $D_k(0,1)$ and $\Berk_k$ will denote the classical closed unit disk and the Berkovich projective line defined over $k$, respectively. 

	We will be occupied for most of this section with the proof of the following result.

\begin{thm}
\label{Thm: Base Change}
	To each extension $K / k$ of algebraically closed and complete non-Archimedean fields, there exists a 
	canonical inclusion map $\iota_k^K : \Berk_k \to \Berk_K$ with the following properties:
	\begin{enumerate}
		\item $\iota_k^K(\zeta_{k,\aa,\rr}) = \zeta_{K,\aa,\rr}$ for each decreasing sequence of closed disks
			$D_k(\aa, \rr) = \left(D(a_i, r_i)\right)_{i \geq 0}$ with $a_i \in k$ and $r_i \in \RR_{\geq 0}$. 
			In particular, $\iota_k^K$ extends the natural inclusion $\PP^1(k) \hookrightarrow \PP^1(K)$ 
			on classical points.
			
		\item If $K' / K$ is a further extension, then $\iota_k^{K'} = \iota_{K}^{K'} \circ \iota_k^K$.
		
		\item $\iota_k^K$ is continuous  for the weak topologies on $\Berk_k$ and $\Berk_K$. 
			In particular, $\iota_k^K(\Berk_k)$ is a compact subset of $\Berk_K$ for the weak topology.
		
		\item $\iota_k^K$ is an isometry for the path-distance metric $\rho$.
		
		\item Write $\iota = \iota_k^K$. 
			For each $x \in \Berk_k$, there exists an injective map $\iota_*:T_x \to T_{\iota(x)}$ such that 
			property that  $\iota(\BB_x(\vec{v})^-) \subset \BB_{\iota(x)}(\iota_*(\vec{v}))^-$  
			for every $\vec{v} \in T_x$. 
	\end{enumerate}
\end{thm}

\begin{remark}
	The map $\iota_k^K$ in the theorem is not a morphism of $k$-analytic spaces except in the case $k=K$, and so we cannot simply appeal to general principles in analytic geometry to determine its properties. Indeed, if it were $k$-analytic, then its construction below would imply the existence of a $k$-analytic morphism of Berkovich disks $\DD_k(0,1) \to \DD_K(0,1)$. Passing to rings of functions, there would exist a $k$-morphism of Tate algebras $K\{T\} \to k\{T\}$. The image of $K$ must lie in a subfield of $k\{T\}$ containing $k$, and so it must be $k$ itself. 
\end{remark}

\begin{remark}
	The existence $\iota_k^K$ and the fact that it is continuous for the weak topology also follow from Poineau's theory of universal points. See \cite[Cor.~3.7, 3.14]{Poineau_Angelique}. More generally, the hypothesis that $k$ is algebraically closed guarantees that the base extension morphism $\pi_{K/k}: \mathsf{X}_K \to \mathsf{X}$ has a continuous section for any analytic space $\mathsf{X}/k$ and any extension of non-Archimedean fields $K/k$. 
\end{remark}

	An important consequence of the continuity properties of the map $\iota_k^K$ is the following application to multiplicities of rational functions.

\begin{cor}
\label{Cor: Compatible Multiplicities}
	Let $K / k$ be an extension of algebraically closed and complete non-Archimedean fields, let $\varphi \in k(z)$ be a nonconstant rational function, and let $\iota = \iota_k^K: \Berk_k \to \Berk_K$ be the inclusion map from the theorem. The following assertions hold:
	
	\begin{enumerate}
		\item If $\varphi_K \in K(z)$ is given by extension 
			of scalars, then $\varphi_K \circ \iota = \iota \circ \varphi$.
			
		\item For every $x \in \Berk_k$, we have  
			$m_{\varphi_K}(\iota(x)) = m_\varphi(x)$. 
			In particular, $\iota^{-1}(\Ram_{\varphi_K}) = \Ram_\varphi$. 
			
		\item For each $x \in \Berk_k$ and each $\vec{v} \in T_x$,  we have
			\[
				m_{\varphi_K}(\iota(x), \iota_*(\vec{v})) = m_\varphi(x, \vec{v}) \ \ \text{ and } \ \ 
			        s_{\varphi_K}(\iota(x), \iota_*(\vec{v})) = s_\varphi(x, \vec{v}).
		    \]
	\end{enumerate}
\end{cor}


\begin{proof}[Proof of Theorem~\ref{Thm: Base Change}]
	Define $\iota_k^K(\zeta_{k, \aa, \rr}) = \zeta_{K, \aa, \rr}$ and $\iota_k^K(\infty) = \infty$. Since cofinality of nested sequences of disks is preserved under base extension, this map is evidently well-defined and injective. Compatibility of the family of maps $\iota_k^{\bullet}$ is clear from the definition. 
	
	For the remainder of the proof, we assume that the extension $K / k$ is fixed and write $\iota = \iota_k^K$ for simplicity. To prove weak continuity of $\iota$, we observe that it suffices to prove $\iota^{-1}(\DD(a,r)^-)$ is open and $\iota^{-1}(\DD(a,r))$ is closed for every $a \in K$ and $r \in \RR_{\geq 0}$. The proof breaks naturally into several cases. 
	
	\textbf{Case 1: $D_K(a,r)^- \cap k \neq \emptyset$.} Without loss of generality, we may assume that $a \in k$. We claim that $\iota^{-1}\left( \DD_K(a, r)^- \right) = \DD_k(a,r)^-$. Suppose first that $\zeta_{k, \bb, \ss} \in \iota^{-1}\left( \DD_K(a, r)^- \right)$. We see that
	\benn
		\ba
			\left|(T-a)(\zeta_{k, \bb, \ss})\right| &= \lim_{i \to \infty} \sup_{x \in D_k(b_i, s_i)} |x-a| \\
				&\leq \lim_{i \to \infty} \sup_{x \in D_K(b_i, s_i)} |x-a| 
				= \left|(T-a)(\zeta_{K, \bb, \ss})\right| < r.
		\ea
	\eenn
Hence $\zeta_{k, \bb, \ss} \in \DD_k(a, r)^-$. 

	For the other containment, suppose that $\zeta_{k, \bb, \ss} \in \DD_k(a,r)^-$. Then $s_i < r$ and $|(T-a)(\zeta_{k, b_i, s_i})| <r$ for $i$ sufficiently large; fix such an $i$ for the moment. For arbitrary $x \in D_K(b_i, s_i)$ and $x' \in D_k(b_i, s_i)$, we see that
	\benn
			|x - a| = |(x - b_i) - (x' - b_i) + (x' - a)| 
				\leq \max \left\{s_i, |(T-a)(\zeta_{k, b_i, s_i})| \right\} < r.
	\eenn
Taking the supremum over all $x \in D_K(b_i, s_i)$ shows $\zeta_{K, b_i, s_i} \in \DD_K(a, r)^-$. Letting $i$ tend to infinity, we see that $\zeta_{K, \bb, \ss} \in \DD_K(a, r)^-$. (Note that $|(T-a)(\zeta_{K, b_i, s_i})|$ is by definition a nonincreasing sequence in the variable~$i$.)

	\textbf{Case 2: $D_K(a,r) \cap k \neq \emptyset$.} The argument here is virtually identical to the previous case. 
If we assume (as we may without loss) that $a \in k$, then  $\iota^{-1}\left( \DD_K(a, r) \right) = \DD_k(a,r)$.

	\textbf{Case 3: $D_K(a,r)^- \cap k = \emptyset$.} We will argue that $\iota^{-1}\left( \DD_K(a,r)^-\right) = \emptyset$. Suppose to the contrary that there exists $\zeta_{k, \bb, \ss}$ such that $|(T-a)(\zeta_{K, \bb, \ss})| < r$. Then for $i$ sufficiently large, we find that
	\[
		|(T - a)(\zeta_{K, b_i, s_i})| = \sup_{x \in D_K(b_i, s_i)} |x-a| < r.
	\]
But then $b_i \in k \cap D_K(a,r)^-$, a contradiction.

	\textbf{Case 4: $D_K(a,r) \cap k = \emptyset$.} Observe that
	\[
		\DD_K(a,r) = \{\zeta_{K, a, r}\} \cup \bigcup_{a' \in D_K(a,r)} \DD_K(a',r)^-.
	\]
We have already shown that the pre-image of each of the latter sets is empty in Case~3, so that $\iota^{-1}\left(\DD_K(a,r)\right) = \iota^{-1}(\zeta_{K,a,r})$. As $\iota$ is injective, we conclude that $\iota^{-1}\left(\DD_K(a,r)\right)$ is either empty or a single point. In either case, it is closed for the weak topology.

	Next,  $\iota$ is an isometry for the path-distance metric because it preserves affine diameters and because it is compatible with the partial orderings on $\Berk_k$ and $\Berk_K$ in the following sense: For every $x, x', y \in \Berk_k$, we have $x \preceq x' \Rightarrow \iota(x) \preceq \iota(x')$ and $\iota(x \vee y) =  \iota(x) \vee \iota(y)$. Indeed, these observations are immediate from the definitions for points of types~I,~II, or~III, and a limiting argument gives them for type~IV points.
	
	By continuity and injectivity, the image of the connected set $\BB_x(\vec{v})^-$ under $\iota$ is connected and does not contain $\iota(x)$. So it must be contained in $\BB_{\iota(x)}(\vec{w})^-$ for some $\vec{w} \in T_{\iota(x)}$. We define $\iota_*(\vec{v}) = \vec{w}$. 
	
	We must show that $\iota_*: T_x \to T_{\iota(x)}$ is injective. This is clear if $x$ is of type~I or type~IV, since $\# T_x = 1$. So we now assume that $x$ is of type~II or~III. Let $\vec{v}_1 \neq \vec{v}_2 \in T_x$. Choose $x_i \in \BB_x(\vec{v}_i)^-$ for $i = 1,2$. It suffices to show $\iota(x_1)$ and $\iota(x_2)$ lie in distinct connected components of $\Berk_K \smallsetminus \{\iota(x)\}$. If $\BB_x(\vec{v}_2)^-$ contains $\infty$, then $x_1 \prec x \prec x_2$. The ordering is compatible with $\iota$, so that $\iota(x_1) \prec \iota(x) \prec \iota(x_2)$. This last inequality impies that $\iota(x_1)$ and $\iota(x_2)$ must lie in distinct connected components of $\Berk_K \smallsetminus \{\iota(x)\}$. By symmetry, we obtain the same conclusion if $\infty \in \BB_x(\vec{v}_1)^-$. Finally, suppose that $\infty \not\in \BB_x(\vec{v}_i)^-$ for $i = 1,2$. In that case, $x_i \prec x$ for $i=1,2$, and $x_1$ and $x_2$ are mutually incomparable under the partial ordering, and we have $x_1 \vee x_2 = x$. Then $\iota(x) = \iota(x_1) \vee \iota(x_2)$, which means $\iota(x_1)$ and $\iota(x_2)$ again lie in distinct connected components of $\Berk_K \smallsetminus \{\iota(x)\}$.
\end{proof}

\begin{proof}[Proof of Corollary~\ref{Cor: Compatible Multiplicities}]
	The first assertion is trivial for type~I points of $\Berk_k$, and the full equality $\varphi_K \circ \iota = \iota \circ \varphi$ follows by weak continuity 
and the fact that type~I points are dense in $\Berk_k$. 
	
	For the second assertion, it evidently holds whenever $x$ is a type~I point by the algebraic description of the multiplicity in that case. Now let $x \in \Berk_k$ be arbitrary, and let $V$ be a $\varphi_K$-saturated weak neighborhood of $\iota(x)$. Then the multiplicity $m = m_{\varphi_K}(\iota(x))$ is equal to $\#V \cap \varphi_K^{-1}(\{y\})$ for each $y \in \varphi_K(V) \cap \PP^1(K)$  (Proposition~\ref{Prop: Top Characterization}).  Now observe that $U = \iota^{-1}(V)$ is a $\varphi$-saturated weak neighborhood of $x$. Since $\varphi$ is defined over the algebraically closed field $k$, we find that $\#U \cap \varphi^{-1}(\{y\}) = m$ for any $y \in \varphi(U)\cap \PP^1(k)$. Thus $m_{\varphi}(x) = m$ as well. 
	
	Finally, let $x \in \Berk_k$, $\vec{v} \in T_x$. Write $\BB_k = \BB_x(\vec{v})^-$ and $\BB_K = \BB_{\iota(x)}(\iota_*(\vec{v}))^-$. Then the proof of the theorem shows $\iota^{-1}(\BB_K) = \BB_k$. The third assertion now follows from Propositions~\ref{Prop: Directional Definition} and~\ref{Prop: Rivera-Letelier Mapping}, the compatibility of $\iota$ and $\varphi$, and what we have already shown in the last paragraph. 	
\end{proof}


\section{The Locus of Inseparable Reduction}
\label{Sec: Inseparable Reduction}

	The phenomenon of inseparable reduction at a type~II point was first investigated by Rivera-Letelier; we spend the present section extending this notion to points of $\Berk$ of arbitrary type. 
In \S\ref{Sec: Ends and Interior} we will characterize the strong interior of the ramification locus in terms of inseparable reduction.

	Let us begin by recalling Rivera-Letelier's definition. Let $\varphi \in k(z)$ be a nonconstant rational function and let $x \in \Berk$ be a type~II point. Let $\sigma_1, \sigma_2 \in \PGL_2(k)$ be chosen so that $\sigma_1(\zeta_{0,1}) = x$ and $\sigma_2(\varphi(x)) = \zeta_{0,1}$. Then $\psi = \sigma_2 \circ \varphi \circ \sigma_1$ fixes the Gauss point, and so $\psi$ has nonconstant reduction $\widetilde{\psi}$. The reduction $\tilde{\psi}$ is well-defined up to pre- and post-composition with an element of $\PGL_2(\tilde k)$. We say that $\varphi$ has \textbf{inseparable reduction} at $x$ if $k$ has positive residue characteristic and $\widetilde{\psi} \in \tilde k(z)$ is inseparable. 
We say that $\varphi$ has \textbf{separable reduction} at $x$ if it does not have inseparable reduction. This definition is stable under extension of scalars:

\begin{prop}
\label{Prop; Stable}
	Let $K / k$ be an extension of complete and algebraically closed non-Archimedean fields, and let $\iota_k^K: \Berk_k \hookrightarrow \Berk_K$ be the canonical inclusion. Then $\iota$ maps type~II points to type~II points, and the function $\varphi$ has inseparable reduction at a type~II point $x \in \Berk_k$ if and only if $\varphi_K$ has inseparable reduction at $\iota_k^K(x)$. 
\end{prop}

\begin{proof}
	If $x = \zeta_{k, a, r}$ is a type~II point, then $r \in |k^\times| \subset |K^\times|$, and hence $\iota_k^K(x) = \zeta_{K,a,r}$ is also a type~II point. We may suppose that $x = \zeta_{k,0,1} = \varphi(\zeta_{k,0,1})$ after a change of coordinate on the source and target. Note that $\iota_k^K$ is compatible with these changes of coordinate (Corollary~\ref{Cor: Compatible Multiplicities}). Since $\widetilde{\varphi_K} = \widetilde{\varphi}_{\tilde K}$, we see that $\varphi$ has inseparable reduction at the Gauss point of $\Berk_k$ if and only if $\varphi_K$ has inseparable reduction at the Gauss point of $\Berk_K$.
\end{proof}

	In order to generalize the definition of inseparable reduction, we will need to know there exist certain kinds of extensions of the field $k$. The following result is well-known, although its proof seems not to be.  

\begin{prop}
\label{Prop: Huge Extension}
	There exists an algebraically closed and complete extension $K / k$ with trivial residue extension such that $K$ is spherically closed and $|K^\times| = \RR_{>0}$. In particular, $\Berk_K$ has no point of type~III or type~IV.
\end{prop}

\begin{proof}
	The construction of a universal field $\Omega_p$ lying over the algebraic closure of $\QQ_p$ given in \cite[pp.137--140]{Robert_p-adic_Book_2000} applies \textit{mutatis mutandis} to our setting. It gives an extension $\hat K / k$ that is algebraically closed and complete, spherically closed, and has the desired value group. However, there is no control over the residue field of $\hat K$ in this construction.
	
	Let $S$ be the set of intermediate extensions of $\hat K / k$ with trivial residue extension. Then $S$ is nonempty, and the union of a linearly ordered collection of elements of $S$ is again an element of~$S$. Zorn's lemma guarantees the existence of a maximal element $K$, which we claim satisfies the conclusion of the proposition.  
		
	Evidently $K$ is complete, since otherwise its completion would be a strictly larger element of~$S$. Next we show that $|K^\times| = \RR_{> 0}$. For otherwise, there exists $r \in \RR_{>0} \smallsetminus |K^\times|$. Let $\AA_r$ be the generalized Tate algebra $K\{r^{-1}T\}$; it is the $K$-algebra of series $f = \sum_{i \geq 0} a_i T^i$ with $K$-coefficients such that $|a_i|r^i \to 0$ as $i \to \infty$. The norm on $\AA_r$ is $\|f\|_r = \sup_{i \geq 0} |a_i| r^i$. Then  $\AA_r$ is a domain, and its fraction field $K_r$ has residue field $\tilde k = \tilde K$ and value group generated by $r$ and $|K^\times|$. Thus $K_r$ contradicts the maximality of $K$. 
	
	Now let $K' / K$ be a finite extension. Since the corresponding extension of residue fields is finite, and since $\tilde K = \tilde k$ is algebraically closed, we see that $\tilde K' = \tilde K$. Hence $K' = K$ by maximality.  
	
	Finally, the spherical closure of $K$ is the maximal extension with the same residue field and value group as $K$. By maximality, we find $K$ is itself spherically closed.  		
\end{proof}

\begin{define}
	Fix a nonconstant rational function $\varphi \in k(z)$. We say that $\varphi$ has inseparable reduction at a type~I point if and only if $\varphi$ is an inseparable rational function. We have already defined above what it means for $\varphi$ to have inseparable reduction at a point of type~II. If $x \in \Berk_k$ is a point of type~III or type~IV, then the preceding proposition shows there exists an extension $K / k$ of algebraically closed and complete non-Archimedean fields such that $\iota_k^K(x)$ is a point of type~II.  We say that $\varphi$ has inseparable reduction at $x$ if $\varphi_K$ has inseparable reduction at $\iota_k^K(x)$. This definition is independent of the choice of field $K$ (Proposition~\ref{Prop; Stable}). 
\end{define}
	
	The notion of inseparable reduction at a type~I or type~II point is evidently intrinsic to the field $k$ by the above definitions. This is also true of type~III points:
	
\begin{prop}
\label{Prop: Type III}
	Suppose $k$ has residue characteristic $p > 0$. Let $\varphi \in k(z)$ be a nonconstant rational function, and let $x \in \Berk$ be a type~III point. Then $\varphi$ has inseparable reduction at $x$ if and only if $p \mid m_\varphi(x)$. 
\end{prop}

\begin{proof}
	Write $m = m_\varphi(x) = m_\varphi(x, \vec{v}_1) = m_\varphi(x, \vec{v}_2)$, where $T_x = \{\vec{v} _1, \vec{v}_2\}$. Let $K$ be an algebraically closed and complete extension of $k$ such that $x_K = \iota_k^K(x)$ is a type~II point of $\Berk_K$. Write $\vec{w}_i = (\iota_k^K)_*(\vec{v}_i)$ for $i = 1,2$. Choose $\sigma_1 \in \PGL_2(K)$ so that \[
	\sigma_1(\zeta_{K,0,1}) = x, \quad (\sigma_1)_*(\vec{0})
		 = \vec{w}_1, \quad (\sigma_1)_*(\vec{\infty}) = \vec{w}_2. 
\] 
Next choose $\sigma_2 \in \PGL_2(K)$ so that 
\[
	\sigma_2(\varphi_K(x)) = \zeta_{K,0,1}, \quad (\sigma_2)_*\left((\varphi_K)_*(\vec{w}_1)\right) = \vec{0}, \quad
	(\sigma_2)_*\left((\varphi_K)_*(\vec{w}_2)\right) = \vec{\infty}. 
\]
Then the map $\psi = \sigma_2 \circ \varphi_K \circ \sigma_1$ satisfies $\psi(\zeta_{K, 0, 1}) = \zeta_{K, 0,1}$, $\psi_*(\vec{0}) = \vec{0}$, and $\psi_*(\vec{\infty}) = \vec{\infty}$. Then $m = m_\psi(\zeta_{K,0,1}) = m_\psi(\zeta_{K,0,1}, \vec{0}) = m_\psi(\zeta_{K,0,1}, \vec{\infty})$  (Corollary~\ref{Cor: Compatible Multiplicities}). The Algebraic Reduction Formula implies that $\widetilde{\psi}(z) = a z^m$ for some nonzero $a \in \tilde K$, and the proof is complete since $p \mid m$ if and only if $\psi$ has inseparable reduction at $\zeta_{K,0,1}$ if and only if $\varphi$ has inseparable reduction at $x$.
\end{proof}


\section{Connected Components}
\label{Sec: Connected Components}

	We open this section by giving a bound on the number of connected components that the ramification locus may have (Theorem~A). Then we study the part of the ramification locus lying off of the connected hull of the critical points. We also give sufficient conditions for when $\Ram_\varphi \subset \Hull(\Crit(\varphi))$. Finally, we show that --- subject to the bound given by Theorem~A --- any number of connected components is achievable. 
	
\begin{prop}
\label{Prop: 2 critical points}
	Let $\varphi \in k(z)$ be a nonconstant rational function. Let $x \in \Berk$ be a point with $m_\varphi(x) > 1$, and let $X$ be the connected component of $\Ram_\varphi$ containing $x$. Then $X$ contains at least $2m_\varphi(x) - 2 \geq 2$ critical points of $\varphi$ counted with weights. 
\end{prop}

\begin{proof}[Proof of Theorem~A] The proposition shows that each connected component of $\Ram_\varphi$ contains at least two critical points, while the Hurwitz formula bounds the number of critical points of a separable rational function by $2\deg(\varphi) - 2$. Hence Theorem~A follows in the separable case. Recall that if $\varphi$ is inseparable, then $\Ram_\varphi = \Berk$ and Theorem~A is trivial. 
\end{proof}


\begin{remark}
	If the characteristic of the field $k$ is positive, then it is possible to have a connected component of $\Ram_\varphi$ containing only one critical point (when counted \textit{without} weight). For example, this is the case for any polynomial function of the form $\varphi(z) = f(z^p) + az$, where $f \in k[z]$ is a nonconstant polynomial and $a \in k$ is nonzero.	
\end{remark}

\begin{proof}[Proof of Proposition~\ref{Prop: 2 critical points}]
	If $\varphi$ is inseparable, then $\Ram_\varphi = \Berk = X$, and $\varphi$ has infinitely many critical points. So the result is trivial in this case. 
	
	Suppose now that $\varphi$ is separable. Let $\{U_\alpha\}$ be the collection of connected components of $\Berk \smallsetminus X$. Note that each $U_\alpha$ is an open Berkovich disk with a type~II endpoint $x_\alpha$ (Proposition~\ref{Prop: Expansion Properties}). Let $\vec{v}_\alpha \in T_{x_\alpha}$ be the tangent direction such that $U_\alpha = \BB_{x_\alpha}(\vec{v}_\alpha)^-$. Then $m_\varphi(x_\alpha, \vec{v}_\alpha) = 1$, since otherwise $U_\alpha \cap X$ would be nonempty. Let $y  = \varphi(x)$. For each index $\alpha$, we apply Corollary~\ref{Rem: Lower Bound} and Proposition~\ref{Prop: Critical surplus} to find that 
	\[
		\#\{\zeta \in U_\alpha : \varphi(\zeta) = y\} \geq s_\varphi(U_\alpha) = \frac{1}{2} \sum_{c \in 
		\Crit(\varphi) \cap U_\alpha} w_\varphi(c).
	\]	
Hence we obtain the estimate
	\begin{align*} 
		\deg(\varphi) = \# \{\zeta \in \Berk : \varphi(\zeta) = y \} 
			&\geq m_\varphi(x) + \sum_\alpha \#\{\zeta \in U_\alpha : \varphi(\zeta) = y\} \\
			&\geq m_\varphi(x) + \frac{1}{2} \sum_\alpha \sum_{c \in \Crit(\varphi) \cap U_\alpha} w_\varphi(c) \\
			&= m_\varphi(x) + \frac{1}{2} \sum_{c \in \Crit(\varphi) \smallsetminus  X} w_\varphi(c).
	\end{align*}
Completing the sum over all critical points and applying the Hurwitz Formula gives 
	\[
		\sum_{c \in \Crit(\varphi) \cap X} w_\varphi(c) \geq \sum_{c \in \Crit(\varphi)} w_\varphi(c)
			+ 2\left[ m_\varphi(x) - \deg(\varphi)\right]		
			 = 2m_\varphi(x) - 2.  \qedhere 
	\]
\end{proof}

	Before describing the part of the ramification locus lying outside the connected hull of the critical points, we need a couple of technical lemmas.

\begin{lem}
\label{Lem: Inseparable reduction}
	Let $\varphi = f / g \in k(z)$ be a nonconstant rational function in normalized form with nonconstant reduction. The following are equivalent:
	\begin{enumerate}
		\item\label{Item: Derivative reduction} $\widetilde{(\varphi')} = 0$
		\item $\widetilde{\wronsk}_\varphi = 0$ (where $\wronsk_\varphi$ is the Wronskian of $\varphi = f / g$)
		\item\label{Item: Inseparable} $\varphi$ has inseparable reduction at the Gauss point
	\end{enumerate}
\end{lem}

\begin{proof}
	The equivalence of the first two statements is immediate since $\wronsk_\varphi$ is the numerator of $\varphi'$. Write $h = \gcd(\tilde{f}, \tilde{g})$, $f_1 = \tilde{f} / h$, and $g_1 = \tilde{g} / h$. Then
		\[
			\widetilde{(\varphi')} = \frac{\tilde{f}' \tilde{g} - \tilde{f} \tilde{g}'}{\tilde{g}^2}
				= \frac{(f_1h' + f_1' h) g_1 h - f_1 h (g_1 h' + g_1' h)}{g_1^2 h^2}
				= \frac{f_1' g_1 - f_1 g_1'}{g_1^2} = (\tilde{\varphi})'.
		\]
Hence we may write $\widetilde{\varphi}'$ without ambiguity. Inseparable rational functions are precisely the kernel of the formal derivative operator; equivalence of \eqref{Item: Derivative reduction} and \eqref{Item: Inseparable} follows.
\end{proof}

\begin{lem}
\label{Lem: First Mapping Lemma}
	Let $\varphi \in k(z)$ be a nonconstant rational function satisfying the following hypotheses:
	\begin{itemize}
		\item $\varphi$ is not injective on the classical disk $D(0,1)^-$, and
		\item $\varphi$ has no critical point in the classical disk $D(0,1)^-$.
	\end{itemize}
Then $0 < p \leq \deg(\varphi)$ and $\varphi$ has inseparable reduction at the Gauss point. 
\end{lem}


\begin{proof}
	We may make a change of coordinates on the target so that $\varphi(\zeta_{0,1}) = \zeta_{0,1}$ and $\varphi_*(\vec{0}) = \vec{0}$. 
	Write $\varphi = f / g$ in normalized form with 
		\begin{eqnarray*}	
			f(z) &=& a_dz^d + a_{d-1}z^{d-1} + \cdots + a_0, \\
			g(z) &=& b_dz^d + b_{d-1}z^{d-1} + \cdots + b_0 ,		
		\end{eqnarray*}
where $a_i, b_j \in k^\circ$. Assume $a_d$ or $b_d$ is nonzero. Write $m = m_\varphi(\zeta_{0,1}, \vec 0)$ and $s = s_\varphi(\zeta_{0,1}, \vec 0)$. The proof of Lemma~\ref{Lem: Surplus Multiplicity} shows $z^{m+s} \mid \mid \tilde{f}$ and $z^s \mid \mid \tilde{g}$. Equivalently, we have
	\begin{align*}
		|a_i| < 1 &\text{ for $0 \leq i \leq m+ s - 1$ and } |a_{m+s}| = 1; \\
		|b_j| < 1 &\text{ for $0 \leq j \leq s - 1$ and } |b_{s}| = 1.
	\end{align*}

	We will now show that the first segment of the Newton polygon of the Wronskian $\wronsk_\varphi$ has negative slope if $p \nmid m$, which is equivalent to saying that $D(0, 1)^-$ contains a root of the Wronskian --- i.e., a critical point of $\varphi$. Evidently this is a contradiction. 
	
	Write $\wronsk_\varphi(z) = \sum c_j z^j \in k^\circ[z]$. From \eqref{Eq: Explicit Wronskian} we see that the constant coefficient of $\wronsk_\varphi$ is $c_0 = a_1b_0 - a_0 b_1$. Since $\varphi$ is not injective on $\DD(0,1)^-$, we find $s+m > 1$, so that both $a_0, a_1 \in k^{\circ \circ}$, which implies $|c_0| < 1$. We also see that the coefficient on the monomial $z^{2s + m - 1}$ is 
	\benn
	\label{Eq: special coefficient}
		c_{2s+m-1} = \sum_{n \neq m + s} ( 2n - 2s - m) a_n b_{2s + m - n} + m a_{s +m} b_s.
	\eenn
We know that $|a_n| < 1$ for $n < s + m$ and that $|b_{2s + m - n}| < 1$ for $2s + m - n < s$, or equivalently when $n > s + m$. So each of the terms in the above sum has absolute value strictly less than~1, while the final term has absolute value $|m|$. If $p \nmid m$, the final term has absolute value~1 and hence dominates the sum. This means  the point $(2s + m - 1, 0)$ lies on the Newton polygon of $\wronsk_\varphi$ (although it may not be a vertex). Hence the first segment of the Newton polygon of $\varphi$ has negative slope. 

	Thus we conclude that $p \mid m \leq \deg(\varphi)$, which gives the desired bounds on the residue characteristic in the lemma. Finally, observe that if any coefficient $c_\ell$ has absolute value~1, then as above we deduce the existence of a critical point of $\varphi$ in the disk $D(0,1)^-$. Thus $|c_\ell| < 1$ for all $\ell \geq 0$. It follows that $\widetilde{\varphi'} = 0$, and an application of Lemma~\ref{Lem: Inseparable reduction} completes the proof. 
\end{proof}

\begin{prop}
\label{Prop: Inseparable Boundary}
	Let $\varphi \in k(z)$ be a nonconstant rational function. Let $U$ be an open Berkovich disk disjoint from $\Hull(\Crit(\varphi))$ with type~II boundary point~$x$. Suppose $U \cap \Ram_\varphi$ is nonempty. Then the following assertions are true: 
	\begin{enumerate}
		\item $0 < p \leq \deg(\varphi)$;
		\item $\varphi$ has inseparable reduction at $x$; and
		\item $U \cap \Ram_\varphi$ is connected (for both the weak and strong topologies).
	\end{enumerate}
\end{prop}

\begin{proof}
	Change coordinates on the source and target so that $x = \varphi(x) = \zeta_{0,1}$ and $U = \DD(0,1)^-$. The assumption $U \cap \Ram_\varphi \neq \emptyset$ implies that $\varphi$ is not injective on $U$. Also, $U$ contains no critical point by hypothesis. Thus $\varphi$ has inseparable reduction at $x$ and $0 < p \leq \deg(\varphi)$ (Lemma~\ref{Lem: First Mapping Lemma}). 	
	
	If $U \cap \Ram_\varphi$ were disconnected, then $U$ would contain an entire connected component of $\Ram_\varphi$. As $U$ contains no critical point, this contradicts Proposition~\ref{Prop: 2 critical points}. 
\end{proof}

Recall from the introduction that a rational function $\varphi$ is \textbf{tame} if its ramification locus has finitely many branch points. 

\begin{cor}
\label{Cor: Large Res Char is Tame}
	Let $\varphi \in k(z)$ be a nonconstant rational function. Suppose that the residue characteristic of $k$ 
	satisfies $p = 0$ or $p > \deg(\varphi)$. Then $\varphi$ is tame. 
\end{cor}

\begin{proof}
	Note $\varphi$ has at least~2 distinct critical points, so that $\Hull(\Crit(\varphi))$ is not reduced to a point. Suppose the result is false, and let $\BB$ be a connected component of $\Berk \smallsetminus \Hull(\Crit(\varphi))$ that meets $\Ram_\varphi$. Then its boundary is of type~II. Proposition~\ref{Prop: Inseparable Boundary} implies that $0 < p \leq \deg(\varphi)$, a contradiction. Hence $\Ram_\varphi \subset \Hull(\Crit(\varphi))$, and $\varphi$ is tame. 
\end{proof}

	We close this section by showing that Theorem~A is optimal. 

\begin{prop}
\label{Prop: n components}
	Let $k$ be an algebraically closed field that is complete with respect to a nontrivial non-Archimedean absolute value. Fix integers $1 \leq n < d$. Then there exists a rational function $\varphi \in k(z)$ of degree~$d$ whose ramification locus $\Ram_\varphi$ has precisely $n$ connected components.
\end{prop}

\begin{proof}
	For the case $n =1$, let $\varphi$ be a polynomial of degree~$d$. Then $m_\varphi(\infty) = d$, and so the connected component $X$ of $\Ram_\varphi$ containing~$\infty$ must contain all of the critical points of $\varphi$ (Proposition~\ref{Prop: 2 critical points}). Any other connected component of $\Ram_\varphi$ would need to contain a critical point, so that $X = \Ram_\varphi$. 
	
	We assume for the remainder of the proof that $n \geq 2$. It will be convenient to set $\ell = n - 1$ and construct  a rational function whose ramification locus has $\ell + 1$ connected components. 
		
	Begin by selecting a rational function $\psi = f / g \in k(z)$ with the following properties:
	\begin{itemize}
		\item $\psi$ has degree~$d - \ell$;
		\item $\widetilde{\psi} \in \tilde k(z)$ is a separable rational function of degree~$d - \ell$; 
		\item $\infty$ is not a critical point for $\psi$;
		\item $\psi = f / g$ is normalized (see \S\ref{Sec: Basics non-Archimedean}); and
		\item $f$ and $g$ are monic of degree~$d - \ell$. 
	\end{itemize}
The set of separable rational functions in $\tilde k(z)$ of degree~$d - \ell$ with simple critical points and non-vanishing leading coefficient in numerator and denominator is a Zariski open subset of the space of all rational functions of degree~$d - \ell$. Choose such a rational function and lift its coefficients to $k^\circ$; if necessary, change coordinate on the source so that $\infty$ is not a critical point. Scaling $f$ and $g$ and perhaps making a scalar change of coordinate on the target allows one to assume $f, g$ are monic. 
	
	Now select elements $a_1, a_2, \ldots, a_\ell \in k^\circ$ with distinct nonzero images in the residue field $\tilde k$. For each $i = 2, \ldots, \ell$, choose $b_i \in k^\circ$ such that $0 < |a_i - b_i| < 1$. Choose $t \in k^{\circ \circ} \smallsetminus \{0\}$. Now we may define a rational function $\varphi \in k(z)$ by
	\[
		\varphi(z) =  \frac{(z-a_1)(z-a_2) \cdots (z-a_\ell)}{(z-b_2) \cdots (z-b_\ell)}\psi(z/t) .
	\]
	
	Evidently the numerator and denominator of $\varphi$ have degree~$d$ and~$d-1$, respectively. To show that $\varphi$ has degree~$d$, we must show that no root of the numerator of $\varphi$ coincides with a root of the denominator. Write 
	\[
		\psi(z) = \frac{z^{d-\ell} + \alpha_{d-\ell-1}z^{d-\ell - 1} + \cdots + \alpha_0}
			{z^{d-\ell} + \beta_{d-\ell-1}z^{d-\ell-1}+ \cdots + \beta_0}.
	\]
Then
	\[
		\psi(z / t) =  \frac{z^{d-\ell} + t\alpha_{d-\ell-1}z^{d-\ell - 1} + \cdots + t^{d-\ell}\alpha_0}
			{z^{d-\ell} + t\beta_{d-\ell-1}z^{d-\ell-1}+ \cdots + t^{d- \ell}\beta_0}.
	\]
A Newton polygon argument shows that the zeros and poles of $\psi(z/t)$ all lie in $D(0,1)^-$. The $a_i$'s and $b_j$'s all have absolute value~1, and $a_i \neq b_j$ for any $i, j$ by construction. Hence $\varphi$ has degree~$d$. 

	The reduction of $\varphi$ is $\widetilde{\varphi}(z) = z - \tilde a_1$. The Algebraic Reduction Formula shows $m_\varphi(\zeta_{0,1}) = 1$, which means that each connected component of $\Ram_\varphi$ lies inside a connected component of $\Berk \smallsetminus \{\zeta_{0,1}\}$. For each $i = 2, \ldots, \ell$, let $U_i$ be the connected component of $\Berk \smallsetminus \{\zeta_{0,1}\}$ containing $a_i$ (and $b_i$). First observe that the surplus multiplicity is $s_\varphi(U_i) = 1$ (Proposition~\ref{Lem: Surplus Multiplicity}). So $U_i$ contains exactly 2 critical points (counted with weights) for $i = 2, \ldots, \ell$ (Proposition~\ref{Prop: Critical surplus}), and hence $U_i$ contains a single connected component of $\Ram_\varphi$ (Proposition~\ref{Prop: 2 critical points}). Set $U_1 = \DD(0,1)^-$. Then  $s_\varphi(U_1) = d - \ell$, so that $U_1$ contains $2(d - \ell)$ critical points (counted with weights). It remains for us to show that $U_1$ contains exactly two connected components of $\Ram_\varphi$.
 	
	Define
		\[
			\eta(z) = \varphi(tz) = \frac{(tz-a_1)(tz-a_2) \cdots (tz-a_\ell)}{(tz-b_2) \cdots (tz-b_\ell)}\psi(z).
		\]
Then $\widetilde{\eta}(z) = (- \tilde a_1) \widetilde{\psi}(z)$, which has degree~$d - \ell$, and so $s_\eta(\zeta_{0,1}, \vec{\infty}) = \ell$ (Lemma~\ref{Lem: Surplus Multiplicity}). The open Berkovich disk $\BB_{\zeta_{0,|t|}}(\vec{v})^-$ contains $2\ell$ critical points, where $\vec{v}$ is the tangent vector corresponding to the connected component of $\Berk \smallsetminus \{\zeta_{0, |t|}\}$ containing $\infty$. We have already accounted for $2(\ell - 1)$ of those critical points above, and so there must be two more critical points --- and hence exactly one more component of $\Ram_\varphi$ ---  in the open annulus $\{x \in \Aff^1 : |t| < |T(x)| < 1 \}$. 

	The reduction of $\eta$ shows $m_\varphi(\zeta_{0, |t|}) = d - \ell$ (Algebraic Reduction Formula). Proposition~\ref{Prop: 2 critical points} shows the connected component of $\Ram_\varphi$ containing $\zeta_{0, |t|}$ also contains at least $2(d-\ell) - 2$ critical points. We have accounted for $2(\ell - 1) + 2 = 2\ell$ critical points in the preceding paragraphs, and we have just located $2(d- \ell) - 2$ more. The Hurwitz formula shows we have now found all of the critical points, and hence all of the connected components of $\Ram_\varphi$. That is, $\Ram_\varphi$ has $\ell+ 1$ connected components. 
\end{proof}


\section{Endpoints and Interior Points}
\label{Sec: Ends and Interior}

	Here we determine the interior and endpoints of $\Ram_\varphi$ for both the weak and strong topologies. We already saw in Proposition~\ref{Prop: Frobenius} that $\Ram_\varphi = \Berk$ if $\varphi$ is itself an inseparable rational function; here we show this is the only case in which the weak interior of $\Ram_\varphi$ is nonempty. Then we characterize the endpoints of the ramification locus and show that the strong interior of $\Ram_\varphi$ coincides with the locus of inseparable reduction. (The definitions were chosen so that this statement holds even when $\varphi$ is inseparable.) We finish the section with a discussion of tame and locally tame rational functions.

\begin{prop}
\label{Prop: Weak Interior}
	The weak interior of the ramification locus of a separable nonconstant rational function is empty.  
\end{prop}

\begin{proof}
	Suppose there exists a rational function $\varphi \in k(z)$ such that the weak interior of its ramification locus is nonempty. Any weak open subset of $\Berk$ contains infinitely many points of type~I, and the type~I points of the ramification locus are precisely the critical points. Thus $\varphi$ has infinitely many critical points, and hence it must be inseparable by the Hurwitz formula. 
\end{proof}

\begin{lem}
\label{Lem: Inseparable Nearby}
	Suppose $k$ has positive residue characteristic $p$, and suppose $\varphi \in k(z)$ is a nonconstant rational function with nonconstant reduction. Let $\vec{v}$ be a tangent direction at the Gauss point of $\Berk$, and write $m = m_\varphi(\zeta_{0,1}, \vec{v})$. Then $p \mid m$ if and only if there exists a point $x \in \BB_{\zeta_{0,1}}(\vec{v})^-$ such that $\varphi$ has inseparable reduction at each point of the segment $(\zeta_{0,1}, x)$.  
\end{lem}

\begin{proof}
	Without loss of generality, we may replace $k$ with an algebraically closed and complete extension in order to assume that $\Berk_k$ has no point of type~III or~IV. (See \S\ref{Sec: Base Change} and \S\ref{Sec: Inseparable Reduction}.)  Moreover, we may change coordinates on the source and target in order to assume that $\vec{v} = \varphi_*(\vec{v}) = \vec{0}$. Write $m = m_\varphi(\zeta_{0,1}, \vec{0})$, and for $t \in k^{\circ \circ} \smallsetminus \{0\}$, define 
	\[
		\varphi_t(z) = t^{-m}\varphi(tz).
	\]
To prove the lemma, it suffices to show that once $\varphi_t$ is properly normalized, it has reduction $\widetilde{\varphi}_t(z) = c z^m$ for some nonzero $c \in \tilde k$ whenever $t \in k^{\circ \circ}$ has absolute value sufficiently close to $1$. Indeed, if $p \mid m$, then this shows $\varphi$ has inseparable reduction at $\zeta_{0, |t|}$. 
	
		We begin by writing $\varphi$ in normalized form as 
		\[
			\varphi(z) = \frac{a_dz^d + \cdots + a_0}{b_dz^d + \cdots + b_0},
		\]
with $a_i, b_j \in k^\circ$ and some coefficient in the numerator and denominator having absolute value~1. Let $s = s_\varphi(\zeta_{0,1}, \vec{0})$ be the associated surplus multiplicity. The Algebraic Reduction Formula and Lemma~\ref{Lem: Surplus Multiplicity} shows that $|a_{m+s}| = |b_s| = 1$, and that $|a_i| < 1$ for $i < m+s$ and that $|b_j| < 1$ for $j < s$. Now observe that
	\be
	\label{Eq: Renormalized}
		\ba
			\varphi_t(z) &= \frac{t^{-m-s}}{t^{-s}} \cdot \varphi(tz) \\
				&= \frac{a_dt^{d-m-s}z^d + \cdots + a_{m+s}z^{m+s} + \cdots + t^{-m-s}a_0}
					{b_dt^{d-s}z^d + \cdots + b_sz^s + \cdots + t^{-s}b_0}.
		\ea
	\ee

	Define $r_0$ to be the maximum element of the set
	\[
		\{|a_i|^{1 / (m+s-i)} : i = 0, \ldots, m+s-1 \} \cup \{|b_j|^{1/(s-j)} : j=0, \ldots, s-1\}.
	\]
If we assume that $r_0 < |t| < 1$, then 
	\be
	\label{Eq: Explicit Coefficients}
		\left|a_it^{i -m-s}\right| 
			\begin{cases} < 1 & \text{if $i \neq m+s$} \\ = 1 & \text{if $i = m+s$} \end{cases}, \qquad
		\left|b_jt^{j-s}\right| 
			\begin{cases} < 1 & \text{if $j \neq s$} \\ = 1 & \text{if $j = s$} \end{cases}.
	\ee
Thus the presentation of $\varphi_t$ given in \eqref{Eq: Renormalized} is normalized, and its reduction is given by $\widetilde{\varphi}_t(z) = (\tilde a_{m+s} / \tilde b_s) z^m$, as desired.
\end{proof}

\begin{prop}[Endpoints of $\Ram_\varphi$]
\label{Prop: Endpoints}
	Let $\varphi \in k(z)$ be a nonconstant rational function, and suppose $x \in \Ram_\varphi$ is an endpoint of the ramification locus. Then $x$ is of type~I,~II, or~IV.  
	\begin{enumerate}
		\item If $x$ is of type~I, then it is a critical point of $\varphi$.
		\item If $x$ is of type~II or~IV, then  $\varphi$ has inseparable reduction at every point 
			of some nonempty segment $(x,y) \subset \Ram_\varphi$. In particular, $0 < p \leq \deg(\varphi)$. 
	\end{enumerate}
\end{prop}

\begin{proof}
	Suppose first that $x \in \Ram_\varphi$ is of type~III, so that it has exactly two tangent directions $\vec{v}_1$ and $\vec{v}_2$. The local degree satisfies $m_\varphi(x, \vec{v}_1) = m_\varphi(x, \vec{v}_2) > 1$ (Proposition~\ref{Prop: Expansion Properties}\eqref{Item: Tangent Properties}). Hence $x$ cannot be an endpoint of $\Ram_\varphi$ (Proposition~\ref{Prop: Directional Definition}\eqref{Item: Locally constant}).
	
	Now let $x \in \Ram_\varphi$ be of type~I. Then $m_\varphi(x) > 1$ is the usual algebraic multiplicity, and hence $x$ must be a critical point of $\varphi$.

	Next suppose that $x$ is a type~II endpoint of $\Ram_\varphi$. After a change of coordinate on the source and target, we may suppose that $x = \zeta_{0,1} = \varphi(\zeta_{0,1})$. Then $\varphi$ has nonconstant reduction at $x$. Since $x$ is an endpoint, we see that $m_\varphi(x, \vec{v}) > 1$ for precisely one tangent direction $\vec{v}$. If $p \nmid m_\varphi(x, \vec{v})$, then the weight of the reduction $\widetilde \varphi$ at $\vec{v}$ satisfies 
	\[
		w_{\tilde \varphi}(\vec{v}) = m_\varphi(x, \vec{v}) - 1 \leq \deg(\varphi) - 1 < 2 \deg(\varphi) - 2,
	\]
in contradiction to the Hurwitz Formula. So $p \mid m_\varphi(x, \vec{v})$, and the result follows upon applying the preceding lemma.

	Now suppose that $x \in \Ram_\varphi$ is of type~IV. Let $y$ be the closest point to $x$ in $\Hull(\Crit(\varphi))$; more precisely, if $U$ is the connected component of $\Berk_k \smallsetminus \Hull(\Crit(\varphi))$ containing $x$, then $y$ is the unique boundary point of $U$. Let $K / k$ be an extension of algebraically closed and complete non-Archimedean fields so that $\Berk_K$ has no point of type~III or~IV. Write $x_K = \iota_k^K(x)$ and $y_K = \iota_k^K(y)$. Proposition~\ref{Prop: Inseparable Boundary} implies that $\varphi_K$ has inseparable reduction at every (type~II) point of the segment $(x_K,y_K)$. Hence $\varphi$ has inseparable reduction at every point of the segment $(x,y)$. 
\end{proof}

\begin{remark}
	When $x$ is an endpoint of $\Ram_\varphi$ of type~II, the induced rational function $\varphi_*: T_x \to T_{\varphi(x)}$ on tangent spaces has a very special property: it is ramified in only one direction. Such rational functions are called unicritical, and were studied in \cite{Faber_One_Critical_Point_2012}. One interesting fact is that the multiplicity at $x$ must satisfy $m_\varphi(x) \equiv 0 \text{ or } 1 \pmod p$. 
\end{remark}

	Rivera-Letelier has characterized when a type~II point lies in the strong interior of the ramification locus:

\begin{prop}[{\cite[Prop.~10.2]{Rivera-Letelier_Periodic_Points_2005}}]
\label{Prop: R-L interior}
	Let $\varphi \in k(z)$ be a nonconstant rational function and let $x \in \Berk$ be a type~II point. Then $\varphi$ has inseparable reduction at $x$ if and only if there exists a strong neighborhood $V$ of $x$ such that $m_\varphi(y) \geq p$ for each $y \in V$. 
\end{prop}

\begin{remark}
	While the result in \cite{Rivera-Letelier_Periodic_Points_2005} is stated over $\CC_p$, the proof is valid for an arbitrary non-Archimedean field (with residue characteristic $p > 0$). Note that the statement is vacuous if $\mathrm{char}(\tilde k) = 0$ (Corollary~\ref{Cor: Large Res Char is Tame}).
\end{remark}

\begin{cor}
\label{Cor: Endpoints and Interior Points}
	Let $\varphi \in k(z)$ be a nonconstant rational function. 
	If $Y$ is a connected component of $\Ram_\varphi \smallsetminus \Hull(\Crit(\varphi))$, 
	then each point of $\overline{Y}$ is either a strong interior point of $\Ram_\varphi$ 
	or an endpoint of $\Ram_\varphi$.  
\end{cor}

\begin{remark}
	As a subspace of $\Ram_\varphi$, the unique relative boundary point of $Y$ will be of type~II in general. However, if $k$ has positive characteristic~$p$, then it is possible for $\varphi$ to have a single critical point (counted without weight), in which case $\Hull(\Crit(\varphi)) = \partial Y$ consists of a single point of type~I. The statement of the corollary applies in either case. 
\end{remark}

\begin{proof}
	Suppose that $y \in \overline{Y}$. If $y$ is of type~I or~IV, it is an endpoint of $\Berk$, and hence also of $\Ram_\varphi$.  If $y$ is of type~II, define $S \subset T_y$ to be the set of tangent directions $\vec{v}$ such that $m_\varphi(y, \vec{v}) > 1$. Then $S$ is nonempty since $\Ram_\varphi$ has no isolated point. If $\#S = 1$, then $y$ is an endpoint. Otherwise, $\#S \geq 2$, and there exists an open Berkovich disk $U$ disjoint from $\Hull(\Crit(\varphi))$ with boundary point $y$ such that $U \cap \Ram_\varphi \neq \emptyset$. Thus $\varphi$ has inseparable reduction at $y$ (Proposition~\ref{Prop: Inseparable Boundary}), and so $y$ is a strong interior point of $\Ram_\varphi$ by the above proposition.
	
	If $y = \zeta_{a,r}$ is of type~III, then we will show it is an interior point of $\Ram_\varphi$. Let $K / k$ be an extension of algebraically closed and complete non-Archimedean fields such that $r \in |K^\times|$, and write $y_K = \iota_k^K(y)$. Then $y_K$ is a type~II point of $\Berk_K$ that lies off of the connected hull of the critical points of $\varphi_K$.  A type~III point can never be an endpoint of the ramification locus; it follows that $y_K$ is not an endpoint of $\Ram_{\varphi_K}$ (Proposition~\ref{Cor: Compatible Multiplicities}). The argument in the previous paragraph applied to $y_K$ and $\varphi_K$ shows that $y_K$ is a strong interior point of $\Ram_{\varphi_K}$. If $V \subset \Ram_{\varphi_K}$ is a strong open neighborhood of $y_K$, then $(\iota_k^K)^{-1}(V) \subset \Ram_\varphi$ is a strong open neighborhood of $y$ (Theorem~\ref{Thm: Base Change}).
\end{proof}

\begin{lem}
\label{Lem: Cone Lemma}
	Suppose $k$ has positive residue characteristic. Let $\varphi \in k(z)$ be such that $s_\varphi(\zeta_{0,1}, \vec{v}) = 0$ for all $\vec{v} \neq \vec{\infty}$, and suppose further that $\widetilde{\varphi}(z) = h(z^p) + cz$ for some nonconstant polynomial $h \in \tilde k[z]$ and some nonzero $c$. Fix $\delta >0$. Then there exists $\varepsilon > 0$ such that $\zeta_{B, |A|} \not\in \Ram_\varphi$ for any $A, B \in k$ satisfying
	\[
		0 < |A| < q_k^{- \delta} \quad \text{and} \quad 1 < |B| < q_k^{\varepsilon}.
	\]
\end{lem}
	
\begin{proof} Let $A, B \in k$ satisfy $0 < |A| < q_k^{-\delta}$ and $|B| > 1$. 
	Set $\psi(z) = A^{-1}\left[ \varphi(Az+B) - \varphi(B)\right]$. If $\varphi(z) = f(z) / g(z)$, then 
	\be
	\label{Eq: Two Terms}
		\psi(z) = \frac{A^{-1}[f(Az+B) - f(B)]}{g(Az+B)} 
			+ \frac{A^{-1}f(B)\left[g(B) - g(Az+B)\right]}{g(B)g(Az+B)}.
	\ee
We will show that the first term above reduces to a linear polynomial in $\tilde k[z]$, and that the second vanishes modulo $k^{\circ \circ}$, provided that $|B|$ is sufficiently close to~1. The Algebraic Reduction Formula then implies $m_\psi(\zeta_{0,1}) = 1 = m_\varphi(\zeta_{B, |A|})$, so that $\zeta_{B, |A|}$ is not in the ramification locus. 

	Write $\varphi$ in normalized form as
	\[
		\varphi(z) = \frac{a_dz^d + \cdots + a_0}{b_dz^d + \cdots +b_0}= \frac{f(z)}{g(z)}.
	\] 
Let $D$ be the degree of the polynomial $h$ in the statement of the lemma. The hypotheses on the surplus multiplicity and on the reduction of $\varphi$ are equivalent to saying $|b_j| < 1$ for $j = 1, \ldots, d$, that $|a_i| < 1$ for $i > Dp$, that $|a_i| < 1$ for $1<i < Dp$ such that $p \nmid i$, and that $|b_0| = 1 = |a_1| = |a_{Dp}|$. 
	
	In the remainder of the proof, we write $\beta$ for any positive real function that tends to zero as $|B| \to 1$, independently of $A$. Note also that if $|A| < q_k^{-\delta}$, then $A$ is uniformly bounded away from~1. Consider the quantity
	\benn
		X_j := A^{-1}\left[ a_j(Az+B)^j - a_jB^j \right] = a_j \sum_{1 \leq i \leq j} \binom{j}{i} A^{i-1}B^{j-i}z^i.
	\eenn
We will show that $\tilde X_j = 0$ for $j \neq 1$ provided $|B|$ is sufficiently close to~1. If $j > Dp$ and $|B|$ is sufficiently close to~1, then $|a_j| < 1$ implies every coefficient of $X_j$ is bounded by $|a_j|(1+ \beta) < 1$. If $1 < j < Dp$ and $p \nmid j$, then each coefficient of $X_j$ is bounded by $|a_j|(1+ \beta) < 1$ for the same reason. If $1 < j \leq Dp$ and $p \mid j$, then 
	\[
		X_j = j a_jB^{j - 1}z + a_jA \sum_{2 \leq i \leq j} \binom{j}{i} A^{i-2}B^{j-i}z^i.
	\]
The linear coefficient has absolute value bounded by $|p|(1+ \beta) < 1$ since $p \mid j$, and the remaining coefficients are bounded by $|A|(1+\beta) < q_k^{-\delta}(1+\beta)$. The remaining cases $j = 0$ and $j = 1$ are treated by observing that $X_0 = 0$ and $X_1 = a_1z$. 

	Next observe that 
	\[
		g(Az+B) - b_0 = \sum_{1 \leq j \leq d} b_j(Az+B)^j.
	\]
Since $|b_j| < 1$ for all $j > 0$, we see that $\widetilde{g(Az+B)} = \tilde b_0$ provided $|B|$ is sufficiently close to~1. 
Hence 
	\benn
		\frac{A^{-1}[f(Az+B) - f(B)]}{g(Az+B)} =
		 \frac{\sum_{0 \leq j \leq d}  X_j}{g(Az+B)} 
		\equiv \frac{ a_1}{ b_0} z \pmod{k^{\circ \circ}}. 
	\eenn
Thus the first term in \eqref{Eq: Two Terms} has the desired reduction.
	
	For the second term in~\eqref{Eq: Two Terms}, we observe that $g(Az+B) = g(B) + A \cdot E(z)$, where $E \in k^{\circ \circ}$ is a polynomial whose coefficients are bounded by $(1+\beta)\max \{|b_j| : j > 0\}$. Note also that $|f(B)| \leq 1+ \beta$. Since $\widetilde{g(B)} = \tilde b_0$, it follows that 
	\[
		\frac{A^{-1}f(B)\left[g(B) - g(Az+B)\right]}{g(B)g(Az+B)} =
			\frac{- f(B)E(z)}{g(B)\left[g(B) + A \cdot E(z) \right]}\equiv 0 \pmod {k^{\circ\circ}}.
	\]
We have now show that the second term in~\eqref{Eq: Two Terms} has the desired reduction when $|B|$ is sufficiently close to~1, which completes the proof.
\end{proof}
	
\begin{prop}
\label{Prop: Interior Point}
	Let $\varphi \in k(z)$ be a nonconstant rational function, and let $x \in \Berk$. Then $\varphi$ has inseparable reduction at $x$ if and only if $x$ is an interior point of $\Ram_\varphi$ for the strong topology.
\end{prop}

\begin{proof}
	First, suppose $x$ is of type~I. By definition, the function $\varphi$ has inseparable reduction at $x$ if and only if $\varphi$ is itself inseparable. In the case that $\varphi$ is inseparable, we have $\Ram_\varphi = \Berk$ (Proposition~\ref{Prop: Frobenius}), so that every classical point is a strong interior point. If $\varphi$ is separable, we must show that $x$ fails to be a strong interior point. A strong open neighborhood of $x$ contains infinitely many type~I points. But the type~I points of $\Ram_\varphi$ are precisely the critical points, of which $\varphi$ has only finitely many. So $x$ cannot be a strong interior point. 

	Now we suppose that $x$ is of type~II,~III, or~IV, and that $\varphi$ has inseparable reduction at $x$. Let $K / k$ be an extension of algebraically closed and complete non-Archimedean fields such that $\Berk_K$ has only type~I and type~II points (Proposition~\ref{Prop: Huge Extension}).  Write $\iota = \iota_k^K$. Then $\iota(x)$ is a type~II point, and Proposition~\ref{Prop: R-L interior} shows that $\varphi_K$ has inseparable reduction at $\iota(x)$ if and only if there exists a strong open neighborhood $V$ of $\iota(x)$ contained inside $\Ram_{\varphi_K}$. By shrinking $V$ if necessary, we may assume it contains no type~I point. Set $U = \iota^{-1}(V)$. Theorem~\ref{Thm: Base Change} and its corollary show that $U \subset \Ram_\varphi$ is a strong open neighborhood of $x$. That is, $x$ is a strong interior point of $\Ram_\varphi$.
	
	For the reverse implication, we assume that $x \in \Berk$ is a strong interior point of $\Ram_\varphi$ and show that $\varphi$ has inseparable reduction at $x$. This is clear by Proposition~\ref{Prop: R-L interior} if $x$ is of type~II. Suppose $x$ is of type~III. The multiplicity $m_\varphi(y)$ is constant with value $m = m_\varphi(x)$ for all type~II points $y$ lying on some segment beginning at $x$  (Propositions~\ref{Prop: Directional Definition} and~\ref{Prop: Expansion Properties}). Now each such point $y$ that is sufficiently close to $x$ in the strong topology must lie in the strong interior of $\Ram_\varphi$. So $\varphi$ has inseparable reduction at $y$; hence $p \mid m_\varphi(y) = m_\varphi(x)$; hence $\varphi$ has inseparable reduction at $x$ (Proposition~\ref{Prop: Type III}). 
	
	Finally, suppose $x$ is a type~IV point in the strong interior of $\Ram_\varphi$. Note that $x$ does not lie on the connected hull of the critical points of $\varphi$. Let $K / k$ be an extension of non-Archimedean fields as in the second paragraph. In particular, $x_K = \iota(x)$ is a type~II point, so it must be either an endpoint or a strong interior point of $\Ram_{\varphi_K}$ (Corollary~\ref{Cor: Endpoints and Interior Points}). In the latter case, $\varphi_K$ has inseparable reduction at $x_K$ (Proposition~\ref{Prop: R-L interior}), and so $\varphi$ has inseparable reduction at $x$ (by definition). 
	
	It remains to show that $x_K$ cannot be an endpoint of the ramification locus of $\varphi_K$. Suppose to the contrary that it is an endpoint. Let $\vec{v} \in T_{x_K}$ be the unique tangent direction such that $m_{\varphi_K}(x_K, \vec{v}) > 1$. We may select $\sigma_1, \sigma_2 \in \PGL_2(K)$ so that $\sigma_1^{-1}(x_K) = \zeta_{K,0,1} = \sigma_2(\varphi_K(x_K))$, and so that $(\sigma_1)_*^{-1}(\vec{v}) = \vec{\infty} = (\sigma_2)_*( (\varphi_K)_*(\vec{v}))$. Set $\psi(z) = \sigma_2 \circ \varphi_K \circ \sigma_1$. Since $x$ is a type~IV point, $m_\varphi(x) = m_{\varphi_K}(x_K) = m_{\varphi_K}(x_K, \vec{v}) > 1$. So $\widetilde{\psi} \in K(z)$ is a rational function that fixes $\infty$, and the (algebraic) multiplicity at infinity equals the degree of $\widetilde{\psi}$. Thus $\widetilde{\psi}$ is a polynomial function. Moreover, $\widetilde{\psi}$ has no finite critical point, and so its  formal derivative must be a nonzero constant $c \in \tilde K$. We conclude that $\widetilde{\psi}(z) = h(z^p) + cz$ for some nonconstant polynomial $h \in \tilde K[z]$. Observe further that $s_\psi(\zeta_{K,0,1}, \vec{w}) = 0$ for all $\vec{w} \neq \vec{\infty}$ since $x_K$ is the image of a type~IV point in $\Berk_k$. 	We are now in a position to apply Lemma~\ref{Lem: Cone Lemma}. 
	
	Recall that we are assuming $x$ is an interior point of $\Ram_\varphi$. Let $\delta_0 > 0$ be such that the $\rho$-ball of radius $\delta_0$ about $x$ lies in $\Ram_\varphi$. Set $\delta = \delta_0 / 3$ and choose $\varepsilon > 0$ as in the lemma. Let $A, B \in K$ be such that (i) $q_k^{-2\delta} < |A| < q_k^{-\delta}$, (ii) $1 < |B| < q_k^{\min\{\varepsilon, \delta_0/6\}}$, and (iii) there exists $y \in \Berk_k$ such that $\zeta_{B, |A|} = \sigma_1^{-1}(\iota(y))$. This last condition is possible because $\sigma_1^{-1}(\iota(\BB_x(\vec{v})^-))$ is a connected subset of $\BB_{\zeta_{K,0,1}}(\vec{\infty})^-$ and shares the same boundary point. Then $y \not\in \Ram_\varphi$ by the lemma. But we also find that
	\[
		\rho(x,y) = \rho(\zeta_{K,0,1}, \zeta_{B, |A|}) =  2\log_{q_k} |B|  - \log_{q_k} |A|
			< 2\log_{q_k} |B| + 2 \delta < \delta_0.
	\]
Hence $y \in \Ram_\varphi$ by our choice of $\delta_0$. This contradiction completes the proof. 
\end{proof}

	Finally, we give a criterion to determine when a rational function is locally tame near a point~$x$ --- i.e.,  when there exists a neighborhood $U$ of $x$ such that $\Ram_\varphi \cap U$ is a finite tree. 

\begin{prop}
	Let $\varphi \in k(z)$ be a nonconstant rational function, and let  $x \in \Berk$. The ramification locus is locally tame near $x$ (for the weak or strong topology) if and only if $p \nmid m_\varphi(x, \vec{v})$ for all tangent vectors $\vec{v} \in T_x$. 
\end{prop}

\begin{remark}
	When $\varphi$ has nonconstant reduction, the proposition says that the ramification locus is locally a finite tree at the Gauss point if and only if the reduction $\widetilde \varphi \in \tilde k(z)$ is tamely ramified. 
\end{remark}

\begin{proof}
	Evidently $p \mid m_\varphi(x, \vec{v})$ for all $x$ and all $\vec{v} \in T_x$ if $\varphi$ is inseparable, and $\Ram_\varphi = \Berk$, so we may exclude this case from the remainder of the proof. We may also assume that $m_\varphi(x) > 1$; else, $\Berk \smallsetminus \Ram_\varphi$ is a weak and strong open neighborhood on which the ramification locus is locally a finite (empty) tree near $x$. 
	
	Suppose first that $p \mid m_\varphi(x, \vec{v})$ for some tangent vector $\vec{v} \in T_x$. Let $U$ be any (weak or strong) open neighborhood of $x$, and let $y \in \BB_x(\vec{v})^- \cap U$ be a type II point such that $m_\varphi(y) = m_\varphi(y, \vec{w}) = m_\varphi(x, \vec{v})$, where $\vec{w}$ is the tangent vector containing $x$. Lemma~\ref{Lem: Inseparable Nearby} implies that $\varphi$ has inseparable reduction at some point of the segment $(x,y)$, so that $\Ram_\varphi \cap U$ is not a finite tree. As $U$ was arbitrary, we conclude that the ramification locus is not locally a finite tree near $x$.
	
	Now suppose that $p \nmid m_\varphi(x, \vec{v})$ for all tangent vectors $\vec{v} \in T_x$. In particular, $\varphi$ has separable reduction at $x$. It suffices to show that $\Ram_\varphi$ is locally a finite tree near $x$ for the weak topology. We claim that there is a weak open neighborhood $U$ of $x$ such that $\varphi$ has separable reduction at all points of $U$. If not, there is a sequence $(y_n)$ of type II points approaching $x$ at which $\varphi$ has inseparable reduction. Since $\varphi$ has separable reduction at $x$, there are only finitely many ramified tangent directions at $x$. It follows that there is a finite set of tangent directions containing the sequence $(y_n)$, else $\Ram_\varphi$ would have infinitely many connected components. By passing to a subsequence if necessary, we may assume that $(y_n)$ lies inside $\BB_x(\vec{v})^-$ for some tangent vector $\vec{v} \in T_x$. Moreover, the hyperbolic distance between $y_n$ and $x$ must tend to zero. There is a path $(x,x') \subset \BB_x(\vec{v})^-$ on which $m_\varphi(y) = m_\varphi(x, \vec{v})$ for $y \in (x,x')$. In particular, $p \nmid m_\varphi(y)$. Thus $y_n \not\in (x,x')$ for any $n$. Since $\Ram_\varphi$ has only finitely many connected components, there must be infinitely many branch points of $\Ram_\varphi$ along $(x,x')$. Each branch must contain a critical point, else Proposition~\ref{Prop: Inseparable Boundary} implies $\varphi$ has inseparable reduction at each branch point. As there are only finitely many critical points, we have reached a contradiction.
	
	To complete the proof, let $U$ be a weak neighborhood of $x$ on which $\varphi$ has separable reduction. For each branch point $y \in U \cap \Ram_\varphi$, each tangent direction $\vec{v} \in T_y$ that points along $\Ram_\varphi$ must contain either a boundary point of $U$ or a critical point of $\varphi$ (Proposition~\ref{Prop: Inseparable Boundary}). Since there are only finitely many of each of these types of point, there can be only finitely many branch points in $U$. 
\end{proof}

	We conclude this section by giving several characterizations of tame rational functions. In particular, this applies when the residue characteristic of $k$ satisfies $p = 0$ or $p > \deg(\varphi)$ (Corollary~\ref{Cor: Large Res Char is Tame}). 
	
\begin{cor}[Tame Characterization]
\label{Cor: Tame characterization}
	Let $\varphi \in k(z)$ be a nonconstant separable rational function. The following statements are equivalent:
	\begin{enumerate}
		\item\label{Item: Tame1} $\varphi$ is tame.
		\item\label{Item: Tame2} $\Ram_\varphi \subset \Hull(\Crit(\varphi))$. 		
		\item\label{Item: Tame3} The ramification locus $\Ram_\varphi$ has empty strong interior.
		\item\label{Item: Tame4} $\varphi$ has separable reduction at all points of $\Berk$.
		\item\label{Item: Tame5} $\varphi$ has separable reduction at all type~II points of $\Berk$.
		\item\label{Item: Tame6} The endpoints of the ramification locus are precisely the critical points of $\varphi$. 
	\end{enumerate}
\end{cor}

\begin{remark}
	With a little more work, one can give another characterization of inseparable reduction that is intrinsic to the field $k$. In the sequel \cite{Faber_Berk_RamII_2012}, we introduce a strong continuous piecewise linear function $\tau_\varphi : \HH \to \RR_{\geq 0}$ --- defined purely in terms of the coefficients of $\varphi$ --- in order to study the behavior of the ramification locus away from the connected hull of the critical points. It turns out that $\tau_\varphi(x) > 0$ if and only if $\varphi$ has inseparable reduction at~$x$. So one could add a further equivalent statement to Corollary~\ref{Cor: Tame characterization}:
	\begin{enumerate}
	\setcounter{enumi}{6}
		\item $\tau_\varphi$ is identically zero on $\HH = \Berk \smallsetminus \PP^1(k)$. 
	\end{enumerate}
\end{remark}

\begin{proof}[Proof of Corollary~\ref{Cor: Tame characterization}] 
	
	\ref{Item: Tame1}. $\Rightarrow$ \ref{Item: Tame2}. Suppose not. Then there is a connected component $U$ of $\Berk \smallsetminus \Hull(\Crit(\varphi))$ such that $U \cap \Ram_\varphi$ is nonempty. By Proposition~\ref{Prop: Inseparable Boundary}, $\varphi$ has inseparable reduction at some point $x \in \Hull(\Crit(\varphi))$, so that $\varphi$ is not locally tame at $x$.  
	
	\ref{Item: Tame2}. $\Rightarrow$ \ref{Item: Tame3}. The separability hypothesis implies $\varphi$ has
		finitely many critical points. 
	
	\ref{Item: Tame3}. $\Rightarrow$ \ref{Item: Tame4}. Proposition~\ref{Prop: Interior Point}.
	
	\ref{Item: Tame4}. $\Rightarrow$ \ref{Item: Tame5}. Clear.	
	
	\ref{Item: Tame5}. $\Rightarrow$ \ref{Item: Tame6}. Proposition~\ref{Prop: Endpoints}. 
	
	\ref{Item: Tame6}. $\Rightarrow$ \ref{Item: Tame1}. Each of the finitely many connected components of $\Ram_\varphi$ is a nontrivial tree, and by hypothesis, the endpoints are precisely the critical points. Finitely many endpoints implies finitely many branch points. 
\end{proof}


\section{The Locus of Total Ramification}
\label{Sec: Total Ram}

\begin{define}
	Let $\varphi \in k(z)$ be a nonconstant rational function. A point $x \in \Berk$ is said to be 
	\textbf{totally ramified} for $\varphi$ if $m_\varphi(x) = \deg(\varphi)$. The \textbf{locus of total ramification} for $\varphi$ is defined as
		\[
			\Ramtot_\varphi = \{x \in \Berk : m_\varphi(x) = \deg(\varphi) \}.
		\]
\end{define}

	 Any map of degree~2 admits a critical point, which must necessarily have multiplicity~2. Thus $\Ramtot_\varphi \neq \emptyset$ when $\deg(\varphi) = 2$. But when $\deg(\varphi) \geq 3$, the locus of total ramification may be empty. 

\begin{thm}
\label{Thm: Totally Ramified}
Let $\varphi \in k(z)$ be a nonconstant rational function. The locus of total ramification $\Ramtot_\varphi$ is a closed and connected subset of the ramification locus $\Ram_\varphi$. If $\Ramtot_\varphi \neq \emptyset$, then $\Ram_\varphi$ is connected and contains $\Hull(\Crit(\varphi))$. In particular, if $\varphi$ is tame and $\Ramtot_\varphi$ is nonempty, then $\Ram_\varphi = \Hull(\Crit(\varphi))$. 
\end{thm}

\begin{proof}
The result is trivial if $\Ramtot_\varphi = \emptyset$ or if $\deg(\varphi) = 1$, so we will assume that we are in neither of these cases in what follows.

Suppose $\zeta \in \Berk$ is totally ramified for $\varphi$. Let $c \in \mathcal{R}_\varphi \smallsetminus \{\zeta\}$, and let $x \in \Berk$ be any point on the open segment $(\zeta, c)$. Then $x$ is of type~II or~III. Write $\BB$ for the open Berkovich disk with boundary point $x$ and containing $c$. Then the image $\varphi(\BB)$ does not contain $\varphi(\zeta)$, and hence cannot be equal to $\Berk$, so  the multiplicities satisfy $m_\varphi(x) \geq m_\varphi(c) > 1$ (Corollary~\ref{Cor: Nonincreasing}). Thus the ramification locus is connected. Taking $c$ to be a critical point of $\varphi$, we also see that every point in the connected hull of the critical points is a ramified point.  This proves the second statement of the theorem. 

Now repeat the argument in the previous paragraph with $c$ a totally ramified point, so that $m_\varphi(x) \geq m_\varphi(c) = \deg(\varphi)$ as well. This proves connectedness of the locus of total ramification. The fact that $\Ramtot_\varphi$ is closed is a consequence of semicontinuity of $m_\varphi$  (Proposition~\ref{Prop: Multiplicity Properties}\eqref{Item: Semicontinuity}).

The final statement follows from Corollary~\ref{Cor: Tame characterization} and what we have already shown. 
\end{proof}

	Let us say that two rational functions $\varphi, \psi \in k(z)$ are \textbf{equivalent} if there exist $\sigma_1, \sigma_2 \in \PGL_2(k)$ such that $\varphi = \sigma_2 \circ \psi \circ \sigma_1$. 

\begin{cor}
\label{Cor: Connected}
Let $\varphi \in k(z)$ be a nonconstant rational function that is equivalent to one of the following:
	\begin{enumerate}
		\item a polynomial or
		\item a map with good reduction (i.e., $\deg(\varphi) = \deg(\widetilde{\varphi})$).
	\end{enumerate}
Then the ramification locus of $\varphi$ is connected and contains $\Hull(\Crit(\varphi))$.  If $\varphi$ is tame,  then $\Ram_\varphi = \Hull(\Crit(\varphi))$.
\end{cor}

\begin{proof}	
	If $\varphi$ is a polynomial, then $\infty \in \PP^1(k)$ is totally ramified, and the theorem applies. If $\varphi$ has good reduction, then the Gauss point is totally ramified for $\varphi$, and we may again use the theorem. The conclusions of the corollary are invariant under change of equivalence class representative (Corollary~\ref{Cor: Coord Change}), so the proof is complete.
\end{proof}

\noindent \textbf{Acknowledgments.}
This work was supported by a National Science Foundation Postdoctoral Research Fellowship. I would like to express my gratitude toward Matt Baker and Bob Rumely for their encouragement during the course of this investigation. J\'er\^ome Poineau and Laura DeMarco deserve my thanks for several helpful discussions. The anonymous referee also made a number of insightful suggestions that improved both the exposition and the content of this article.

\bibliographystyle{plain}
\bibliography{xander_bib}

\end{document}